\documentclass[reqno]{amsart}

\pdfoutput=1
\usepackage{general,strings,cancel,txfonts}
\usepackage{enumitem}
\usepackage[bookmarks=false, colorlinks=true, pdfstartview=FitV, linkcolor=blue, citecolor=blue, urlcolor=blue]{hyperref}
\newcommand{\dt}[1][t]{\,\operatorname{d}\negsp[2]#1} %measure
\newcommand{\de}[1][\ge]{\,\operatorname{d}\negsp[2]#1} %measure
\newcommand{\dgt}[1][\gt]{\,\operatorname{d}\negsp[2]#1} %measure
\newcommand{\inn}{\operatorname{int}}

\newcommand{\bd}{\operatorname{bd}}

\renewcommand{\simt}{S}
\newcommand{\simS}{\ensuremath{\mathbf{S}}\xspace}
\newcommand{\simtset}{\ensuremath{\{\simt_1,\ldots,\simt_N\}}\xspace}
\newcommand{\Words}{\sW}     %set of finite words
\newcommand{\Mink}{\sM}

\newcommand{\Minka}{\ensuremath{\cj{\Mink}}\xspace}
\newcommand{\one}{\mathbbm{1}}

\newcommand{\ee}{\operatorname{e}}

\newcommand{\defeq}{:=}%{\vcentcolon=}
\newcommand{\eqdef}{=:}%{=\vcentcolon}

\renewcommand{\tilde}{\widetilde}

  \let\oldmarginpar\marginpar
  \renewcommand\marginpar[1]
    {\-\oldmarginpar[\raggedleft\tiny #1]%
    {\raggedright\tiny #1}}

\numberwithin{theorem}{section}
\numberwithin{equation}{section}

\allowdisplaybreaks

\title[Lattice nonmeasurability in the pluriphase case]
{Lattice-type self-similar sets with pluriphase generators fail to be Minkowski measurable}

\author{Sabrina Kombrink}
  \address{Universit\"at Bremen, FB 3 -- Mathematik, Bibliothekstr.1, 28359 Bremen, Germany}
  \email{\href{mailto:kombrink@math.uni-bremen.de}{kombrink@math.uni-bremen.de}}

\author{Erin P. J. Pearse}
  \address{California Polytechnic State University, Mathematics Department, San Luis Obispo, CA 93407-0403 USA}
  \email{\href{mailto:epearse@calpoly.edu}{epearse@calpoly.edu}}

  \author{Steffen Winter}
  \address{Karlsruhe Institute of Technology, Department of Mathematics, 76128 Karlsruhe, Germany}
  \email{\href{mailto:steffen.winter@kit.edu}{steffen.winter@kit.edu}}


    \thanks{Version of \today} %\textbf{\currenttime} on  \textbf{\longdate{\today}}.
\begin{document}
%\linenumbers

\begin{abstract}
	A long-standing conjecture of Lapidus claims that under certain conditions, self-similar fractal sets fail to be Minkowski measurable if and only if they are of lattice type. % (i.e., the scaling ratios in the defining iterated function system satisfy a certain type of arithmetic regularity). % \cite{Lap:Dundee}.
 The theorem was established for fractal subsets of \bR by Falconer, Lapidus and v.~Frankenhuijsen, and the forward direction was shown for fractal subsets of \bRd, $d \geq 2$, by Gatzouras. % \cite{Gat}. 
Since then, much effort has been made to prove the converse. In this paper, we prove a partial converse by means of renewal theory. Our proof allows us to recover several previous results in this regard, but is much shorter and extends to a more general setting; several technical conditions appearing in previous versions of this result have now been removed.
\end{abstract}

\maketitle

\section{Introduction}\label{sec:intro} 

We address the Minkowski measurability (see Def.~\ref{def:Minkowski-measurable}) for self-similar sets in $\mathbb R^d$. In particular, we attempt to characterize this property in terms of the lattice properties of the underlying iterated function system (IFS). Let $S=\{S_1,\dots,S_N\}$ denote an IFS in which each $S_i$ is a contractive similarity acting on $\mathbb R^d$, called a \emph{self-similar system} in the sequel. Further, let $F\subseteq\mathbb R^d$ denote the self-similar set which is the unique non-empty compact set satisfying $F = \bigcup_{i=1}^N S_i F$; see \cite{Hut}. If the scaling ratio of $S_i$ is denoted by $r_i$, then the IFS is said to be \emph{lattice} if there is an $r>0$ such that each $r_i$ can be written as $r^{k_i}$ for some integer $k_i \in \mathbb N$ (see Def.~\ref{def:lattice}), otherwise, the IFS is said to be \emph{nonlattice}.

For $d=1$, it was proven (see \cite{Falconer95,FGCD,KessebohmerKombrink13}) under very general conditions that $F$ is Minkowski measurable if and only if the IFS is nonlattice. These conditions include that the Minkowski dimension $D$ of $\attr$ lies strictly between 0 and 1, %see Def.~\ref{def:nontriv})
 and that the IFS satisfies the open set condition (OSC), see Def.~\ref{def:OSC}.
It was conjectured in \cite[Conj.~3]{Lap:Dundee} that this equivalence statement should remain true for $d\geq 2$, when $d-1<D<d$. In \cite{Gat}, Gatzouras was able to prove and strengthen (under the OSC) one direction of this conjecture, namely, that for arbitrary $d \in\mathbb N$ and $D\in(0,d)$, the self-similar attractor $F\subset\mathbb R^d$ is Minkowski measurable when the IFS is nonlattice. 
It is an open problem to prove the converse, and this has been a very active area; see, for example, \cite{TFCD, Pointwise, Minko, RatajWinter2, Kombrink:thesis, Kocak1}. Our results in this paper give some further progress towards establishing the converse, i.e., showing that the attractor of a lattice self-similar system is not Minkowski measurable. At the same time we demonstrate that 
it is essential to exclude sets of integer Minkowski dimension $D$ from the conjecture but that it is plausible to extend Lapidus' conjecture to the setting of non-integer $D\in(0,d)$; see Rem.~\ref{thm:trivial-issues}.
%there are further exceptions to the conjecture than just full-dimensional sets for which the conjecture obviously fails. In particular, the situation for sets of integer Minkowski dimension turned out to be delicate.

 We work in a setting which includes -- to the best of our knowledge -- all the previous cases in which the Minkowski measurability of lattice self-similar sets has been addressed (see the detailed discussion at the very end of the introduction) and which extends the class of sets covered in several directions. For instance, we do not require the set $F$ to possess a \emph{compatible} feasible open set $O$ satisfying the OSC, that is, one which satisfies $\bd O\subset F$. This allows in particular to treat self-similar sets of any Minkowski dimension $D$ and removes the assumption $D>d-1$, which is present in all previous work known to the authors.
  Instead of compatibility, we will assume throughout that the feasible open set $O$ we work with  satisfies the following additional conditions:
\begin{itemize}
	\item (Strong OSC)  $O \cap F \neq \es$;
	\item (Projection condition) $S_i O \subseteq \overline{\pi_{F}^{-1}(S_iF)}$ for $i=1,\ldots,N$.
\end{itemize}
Here $\pi_F$ denotes the metric projection onto $F$ (see Def.~\ref{def:metproj}) and $\overline{A}$ denotes the closure of $A\subset\mathbb R^d$.
  It follows from results in  \cite{BandtNguyenRao} that one can always find a feasible open set satisfying both the strong OSC (SOSC) and the projection condition whenever OSC is satisfied (the ``central open set''; see Rem.~\ref{rmk:centralopen}). Therefore, these two conditions alone do not restrict at all the class of self-similar sets considered; they should be seen as a convenient choice of feasible open set we make in order to simplify the problem. %(Note that compatible open sets satisfy automatically both, they are strong and satisfy the projection condition. They are thus admissable open sets in our setting.)
  %Equivalence of OSC and SOSC has previously been shown in \cite{Schief}.
%
The only further (rather restrictive) assumption we require is the following. We suppose that the $\ge$-parallel set $F_{\ge}$ of $F$ (see \eqref{eqn:A_ge}) is well-behaved in the set
\linenopax
\begin{align}
  \label{eq:Gamma1}\Gamma=\Gamma(O)\defeq O\setminus \bigcup_{i=1}^N S_i O.
\end{align}
More precisely, it is required that the parallel volume $\lambda_d( F_\varepsilon\cap \Gamma)$ of $F$ restricted to the set $\Gamma$ (where $\gl_d$ is Lebesgue measure) is piecewise polynomial in the variable $\varepsilon$. In this case we call the set $F$  \emph{pluriphase} with respect to $\Gamma$ (and we call $F$ \emph{monophase} with respect to $\Gamma$ if this parallel volume is a polynomial), see Def.~\ref{def:pluriphase} and the discussion afterwards for details.
The pluriphase condition is a simplifying assumption on the geometry of $F$ which would ideally be removed in future work. Note: the definitions of pluriphase and monophase are extended here to the 
present more general setting. In case of a compatible feasible set it reduces to the pluriphase/monophase assumption made in earlier work on this topic, e.g. in \cite{Kocak1, Kombrink:thesis, Minko}.). See also Rem.~\ref{rmk:pluriphase:generator}.

The main results of this paper are summarized in the following statement.

\begin{theorem}\label{thm:main-result}
	Let $F\subset\bR^d$ be a self-similar set which is the attractor of a lattice self-similar system $S=\{S_1,\dots,S_N\}$, $N\geq 2$ satisfying the OSC. Let $D:=\dim_{\mathcal M}F$ denote its Minkowski dimension.
 \begin{enumerate}
   \item\label{it:intro:fulldim} If $D=\dim \operatorname{aff} F$ (where the latter is the dimension of the affine hull of $F$), then $F$ is Minkowski measurable. In particular, this is true for $D=d$.
   \item\label{it:intro:noninteger} Suppose $D<d$ is not an integer and there exists a strong feasible set $O$ satisfying the projection condition such that $F$ is pluriphase with respect to the set $\Gamma(O)$. Then $F$ is not Minkowski measurable.
   \item\label{it:intro:integer} Suppose $D<d$ is an integer and there exists a strong feasible set $O$ satisfying the projection condition such that $F$ is pluriphase with respect to the set $\Gamma(O)$. Then $F$ is Minkowski measurable if and only if certain algebraic relations involving the data of the pluriphase representation are satisfied. In particular, these relations are never satisfied in the case when $F$ is monophase with respect to $\Gamma(O)$.
        \end{enumerate}
\end{theorem}

%old version of same Theorem
%\begin{theorem}\label{thm:main-result}
%	Let $F$ be a self-similar set which is the attractor of the lattice self-similar system $S=\{S_1,\dots,S_N\}$, $N\geq 2$ and which satisfies the OSC. Let $\dim_{\mathcal M}F$ denote its Minkowski dimension and assume that $\dim_{\mathcal M}F<d$. Further, let $O$ be a feasible open set which satisfies both the SOSC and the projection condition as well as the following additional conditions:
%	\begin{enumerate}[label=(\roman*)]
%		%\item (Nontrivial; see Def.~\ref{def:nontriv}) \; $O \nsubseteq \bigcup_{i=1}^N \cj{S_i(O)}$.
%		\item\label{it:into:pluri} (Pluriphase; see Def.~\ref{def:pluriphase}) \;
%		 $(F,O)$ is pluriphase in the sense described just above.
%		\item\label{it:intro:integer} In the case when $\dim_\sM F$ is an integer, a further condition excluding trivial self-similar sets of lower dimension is necessary.
%  \setcounter{enumcontinue}{\value{enumi}}
%	\end{enumerate}
%	In this situation, $F$ is not Minkowski measurable.
%\end{theorem}
%The lattice condition is part of the context of the general conjecture.

Part \ref{it:intro:noninteger} is reformulated and proved in Thm~\ref{thm:pluriphase-result}; a more precise formulation of part \ref{it:intro:integer}, the case when $D$ is an integer, is given in Thm~\ref{thm:pluriphase-integer-result}.
We stress that in the situation of part \ref{it:intro:integer} both cases are possible: lattice sets in $\bR^d$ of integer Minkowski dimension $D<d$ can be Minkowski measurable or not; see Rem.~\ref{thm:trivial-issues}.

As for part \ref{it:intro:fulldim}, we can give a short proof immediately.%, but it uses the notion of \emph{(\ga-dimensional) Minkowski content}, which is denoted $\sM_\ga$. The reader unfamiliar with this concept can find it in Definition~\ref{def:Minkowski-measurable}.

\begin{proof}[Proof of Thm.~\ref{thm:main-result}\ref{it:intro:fulldim}]
%First, we consider the case when $\dim_{\sM} F = d$. 
As a function of \ge, the tubular volume $\gl_d(F_\ge)$ is continuous and strictly increasing on $(0,\iy)$ for any compact set $F \ci \bRd$, and thus $\sM_d(F) := \lim_{\ge \searrow 0} \gl_d(F_\ge) = \gl_d(F)$; see Def.~\ref{def:Minkowski-measurable} for the full definition of Minkowski content $\sM_d$. In other words, the limit must exist and it is not difficult to see that it coincides with $\gl_d(F)$. Now, if $F$ is a self-similar set with $\dim_{\sM} F=d$ satisfying OSC, then it is well known that $F$ has interior points (see e.g. \cite{Schief, GeometryOfSST}) and therefore $\sM_d(F) = \gl_d(F)>0$ . Thus $F$ is Minkowski measurable as claimed.  
Now, if $F \ci \bRd$ is a self-similar set such that $\dim_{\sM} F = \dim \operatorname{aff} F$, then Minkowski measurability follows from working in the affine hull and observing that Minkowski measurability is independent of the dimension of the ambient space; cf.\cite{Resman,Kneser55} and the references therein.
\end{proof}

\begin{remark}\label{thm:trivial-issues}
	The above proof indicates that if $\dim_{\sM} F = \dim \operatorname{aff} F$, then $F$ is Minkowski measurable regardless of whether it is lattice or nonlattice.
	The significance of this point is as follows: one may be naturally led to suppose that the original conjecture of Lapidus may be extended to include sets with Minkowski dimension $D \in (0,d)$ (instead of requiring $D \in (d-1,d$), but Thm.~\ref{thm:main-result}\ref{it:intro:fulldim} shows there are (somewhat trivial) counterexamples with $\dim_{\sM} F = \dim \operatorname{aff} F$.
	%It clarifies that only for (integer dimensional) lattice sets that need more space than their Minkowski dimension suggests Minkowski measurability may fail excludes the ``trivial cases'' from the discussion in part (iii).
	For the case when $D<d$ is an integer (as in Thm.~\ref{thm:main-result}\ref{it:intro:integer}), a class of examples of Minkowski measurable lattice sets is discussed in Sec.~\ref{sec:examples}, Ex.~\ref{ex:square}. At this point, all known examples of this type can be represented as the embedding in $\bR^d$ of a self-similar set in $\bR^D$. It may well be that this is the only way such a thing is possible.
\end{remark}

The proofs of the other parts of Thm.~\ref{thm:main-result} are obtained using the elementary tools of probabilistic renewal theory. %, instead of the more powerful (but also more complicated) apparatus of \emph{fractal strings and complex dimensions} \cite{FGCD}. 
These are combined with recent results in \cite{MinkContGenForm}, where the projection condition was observed to be essential for deriving renewal equations in terms of the generator of an associated tiling.  %This new approach simplifies the proofs considerably compared to previous work and allows to work in a generalizes .
Based on the renewal theorem, we obtain in Thm.~\ref{thm:subresult} and Cor.~\ref{cor:subresult}  a characterization of Minkowski measurability in terms of a periodic function $p$ (see \eqref{eqn:p(e)}) which can be expressed completely in terms of the parallel volume $\lambda_d(F_\varepsilon\cap \Gamma)$ of $F$ restricted to $\Gamma$. This result may be of independent interest for future work as it does not require the pluriphase condition and thus applies to all (nontrivial) self-similar sets. We use this result to prove our main results, the non-Minkowski measurability in the case when $F$ is pluriphase w.r.t.\ the set $\Gamma$, by exploiting and refining an idea in \cite{Kombrink:thesis}.

Before moving on to the results, we explain in more detail the improvements obtained here compared to previous results from \cite{Kocak1,Kombrink:thesis,Minko}. The assumptions of \cite{Kombrink:thesis} and \cite{Minko} only differ slightly, and combining \cite[Thm.~2.38]{Kombrink:thesis} with \cite[Thm.~5.4]{Minko} the resulting nonmeasurability result applies to the case when the following requirements are met:

\begin{enumerate}[label=(\alph*)]
%\item \label{item:abscissa} The Minkowski dimension $D$ of $F$ satisfies $d-1 < D < d$.
\item \label{item:OSC} The open set condition (see Def.~\ref{def:OSC}) holds for a feasible open set $O$ with $\bd O \ci F$, which in particular implies $d-1 < D < d$.
\item \label{item:outer-Minkowski} The $D$-dimensional outer Minkowski content of $\cj{O}$ is finite (see \cite[Def.~5.2]{Minko}).
\item \label{item:one-gen} The \emph{generator} $G = O \less \bigcup_{i=1}^N \cj{S_i O}$ (see Def.~\ref{def:self-similar-tiling}) has only finitely many connected components.
\item \label{item:monophase} Each connected component of the generator $G$ is \emph{monophase} (see Def.~\ref{def:pluriphase}).
\end{enumerate}
In our main result, Thm.~\ref{thm:main-result}, we remove \ref{item:OSC}--\ref{item:one-gen}, and replace \ref{item:monophase} with the more general condition that $\attr$ is pluriphase w.r.t.\ $\gG$. 
The significance of removing \ref{item:OSC} is that the results of the present paper (in particular, Thm.~\ref{thm:main-result}) cover sets of any Minkowski dimensions $D \leq d$. %Moreover, \ref{item:outer-Minkowski} is automatically satisfied, since the $D$-dimensional outer Minkowski content of $\cj{O}$ is equal to 0 in the present context by\cite[Thm.~3.11]{MinkContGenForm}.
Other notable improvements of the present article are the comparatively shorter and simpler proofs based on probabilistic renewal theory. In \cite{Minko}, the more powerful (but also more complicated) apparatus of \emph{fractal sprays and complex dimensions}  was used; the approach taken in \cite{Kombrink:thesis} uses renewal theory in symbolic dynamics (motivated by \cite{Lalley}) yielding results for the more general class of self-conformal sets. When the renewal theorem for symbolic dynamics is restricted to the self-similar setting, it boils down to the probabilistic renewal theorem (as used in \cite{Lalley88}) which we apply directly here. This direct application makes the proofs significantly shorter and simpler.
%In conjunction with the recent results of \cite{MinkContGenForm} which allow results for the tiling to be transferred to the fractal itself, the overall effect generalizes and simplifies the theory considerably.
It should be noted that under the additional assumption that $O$ coincides with the interior of the convex hull of $F$, a result similar to \cite[Thm.~5.4]{Minko} / \cite[Thm.~2.38]{Kombrink:thesis} was independently proven in \cite{Kocak1} using Mellin transforms. See \cite{Kocak2013a,Kocak2013b,Kocak2014} for further interesting and related results.

The structure of the article is as follows. In Sec.~\ref{sec:prelim} we lay the foundations for stating and proving our main results in Sec.~\ref{sec:results}. The final section, Sec.~\ref{sec:examples}, is devoted to examples demonstrating our findings.

\section{Preliminaries}\label{sec:prelim}

We present the terminology required to state and prove our main theorems. %Whenever possible, notation is chosen to be compatible with that of \cite{Pointwise,Minko}.

\subsection{Minkowski measurability}\label{sec:Minkmb}
%\begin{defn}\label{def:parallel}
	Let $A$ be a compact subset in Euclidean space $\bR^d$ and $\ge\geq 0$. The \emph{$\ge$-parallel set} of $A$ (or  \emph{$\ge$-tubular neighborhood} of $A$) is
	\linenopax
	\begin{align}\label{eqn:A_ge}
	  A_\ge\coloneqq\{x\in \bR^d: d(x,A)\leq \ge\},
	\end{align}
	where $d(x,A)\coloneqq\inf\{\|x-a\|:a\in A\}$ is the Euclidean distance of $x$ to the set $A$.
	
	A \emph{tube formula} for $A$ is an explicit formula for $\lambda_d(A_\ge)$, as a function of \ge, where $\lambda_d$ denotes the $d$-dimensional Lebesgue measure; see \cite{TFCD,Pointwise,Minko} for a discussion of fractal tube formulas. The volume $\lambda_d(A_\ge)$ is referred to as the \emph{$\ge$-parallel volume} of $A$ and we call $\lambda_d(A_\ge\cap B)$ the \emph{$\ge$-parallel volume of $A$ inside $B$} for any Borel set $B\ci\mathbb R^d$.
%\end{defn}

\begin{defn}\label{def:Minkowski-measurable}
	Let $A$ be a compact subset of Euclidean space $\bR^d$. For $0\le \alpha\le d$, we denote by
	\linenopax
	\begin{align} \label{eqn:M-content}
	  \sM_\alpha(A)\coloneqq\lim_{\ge\to 0^+}\ge^{\alpha-d}\lambda_d(A_\ge)
	\end{align}
	the \emph{$\alpha$-dimensional Minkowski content} of $A$ whenever this limit exists (as a value in $[0,\infty]$).

	If $\sM_\alpha(A)$ exists and satisfies $0<\sM_\alpha(A)<\infty$, then $A$ is called \emph{Minkowski measurable (of dimension $\alpha$)}, and $\dim_\sM A\coloneqq \ga$ is the \emph{Minkowski dimension} of $A$.
	If the limit in \eqref{eqn:M-content} does not exist, one may consider the logarithmic Ces\`{a}ro average known as the \emph{($\alpha$-dimensional) average Minkowski content} (which always exists in the case of self-similar sets $A$, see \cite{Gat}). It is defined by
	\linenopax
	\begin{align*}%\label{eqn:}
		\Minka_\ga(A) \coloneqq \lim_{\gd\to 0^+} \frac1{|\ln \gd|} \int_\gd^1 \ge^{\ga-d} \gl_d(A_\ge) \frac{\de}{\ge}.
	\end{align*}
	whenever this limit exists.
\end{defn}

\subsection{Self-similar tilings and their generators}\label{sec:terms}

Let $S = \simtset$, $N \geq 2$ be an iterated function system (IFS), where each $\simt_i$ is a  similarity mapping of \bRd with scaling ratio $r_i$, where $0 < r_i < 1$.
Then we call $S$ a \emph{self-similar system}.
For $A \ci \bRd$, we write
\linenopax
\begin{align}\label{eqn:S-action}
	\simS A \coloneqq \bigcup_{i=1}^N \simt_i(A).
\end{align}
The \emph{self-similar set} $\attr$ generated by the IFS $S$
 is the unique compact and nonempty solution of the fixed-point equation $\attr = \simS \attr$\,; cf.~\cite{Hut}, also called the \emph{attractor} of $S$.
%
%, i.e.,
%\linenopax
%\begin{align*}%\label{eqn:}
%  \simt_n(x) = r_n M_n(x) + t_n,
%\end{align*}
%where $M \in O(d)$ is a rotation and/or reflection, $r_n \in (0,1)$ is a contractive scaling ratio, and $t_n \in \bRd$ is a translative component (i.e. a vector).

We study the parallel volume of the attractor by studying the parallel volume inside a certain tiling of its complement, which is constructed via the IFS as described below. The tiling construction was introduced in \cite{SST} and developed in \cite{GeometryOfSST}, where tilings by open sets were studied; see also \cite{Pe2,TFCD,Pointwise,Minko}. In this paper we consider self-similar tilings in a generalized sense, with the tiles not necessarily being open (see Def.~\ref{def:self-similar-tiling}).
%
%via its {tube formula}.
%
%\begin{defn}\label{def:tube-formula}
%  The \emph{inner parallel volume} of a bounded subset $A \ci \bRd$ is the function
%\linenopax
%\begin{align*}%\label{eqn:}
%  V(A,\ge) := \gl_d\left(\{x \in A \suth \dist(x, A^c) \leq \ge\}\right),
%  \qq \ge \geq 0,
%\end{align*}
%where $\gl_d$ is $d$-dimensional Lebesgue measure and $\dist(x, A^c) = \inf_{y \in \bRd \less A} |x-y|$ refers to the usual Euclidean distance.
%  A \emph{tube formula} is an explicit expression for $V(A,\ge)$.
%\end{defn}
%
%We are interested in tube formulas of \emph{self-similar tilings}.
%Such a set (or rather, collection of sets) is a partial decomposition of the complement of \attr which These tilings have been studied in some detail in \cite{SST,GeometryOfSST,TFCD, Demir, Kocak2014, KocakNote, Kocak1, Kocak2}. A formula for the volume of the inner tubular neighborhood of this tiling was given in \cite{Pointwise}, and we now investigate some of the implications of this formula for the Minkowski measurability of the attractor (as defined precisely in \S\ref{sec:Minkowski-measurability}).
%
The construction of a self-similar tiling requires the IFS to satisfy the \emph{open set condition} and a \emph{nontriviality condition}, as described in the following two definitions.
%The open set condition is well-studied, see , for example. The nontriviality condition and the following two propositions are excerpted from \cite{GeometryOfSST}.

\begin{defn}\label{def:OSC}
  A self-similar system $S=\simtset$ %(or its attractor \attr)
satisfies the \emph{open set condition} (OSC) if and only if there is a nonempty open set $O \ci \bRd$ such that
  \linenopax
  \begin{equation}\label{eqn:def:OSC-containment}%\label{eqn:def:OSC-disjoint}
  \begin{array}{ll}
    \displaystyle{\simt_i(O) \ci O}, & i=1, \dots, N \q\text{and}\\
    \displaystyle{\simt_i(O) \cap \simt_j(O) = \es}, &  i \neq j.      
  \end{array}
  \end{equation}
  In this case, $O$ is called a \emph{feasible open set} for $\simtset$% (or \attr)
; see \cite{Hut, Falconer, BandtNguyenRao}. If additionally $O\cap F\neq\es$, then $O$ is called a \emph{strong} feasible open set.
\end{defn}
%\begin{remark}
	It was shown in \cite{Schief} that if a self-similar system satisfies OSC, then it possesses a strong feasible open set.
%\end{remark}

\begin{defn}\label{def:nontriv}
  A self-similar set \attr, which is the attractor of a self-similar system $S=\simtset$ satisfying OSC, is said to be \emph{nontrivial} if there exists a feasible open set $O$ such that
  \begin{equation}\label{eqn:nontriv}
    O\not \ci \cj{\simS O},
  \end{equation}
  where $\cj{\mathbf{S}O}$ denotes the closure of $\mathbf{S}O$; otherwise, \attr is called \emph{trivial}.
\end{defn}

This condition is needed to ensure that the set $\gG = O \setminus \simS O$ in Def.~\ref{def:self-similar-tiling} has nonempty interior.  It turns out that nontriviality is independent of the particular choice of the set $O$. It is shown in \cite{GeometryOfSST} that \attr is trivial if and only if it has nonempty interior, %interior points
which amounts to the following characterization of nontriviality:

\begin{prop}[{\cite[Cor.~5.4]{GeometryOfSST}}]
  \label{cor:OSC-dimension-d-implies-trivial}
  Let $\attr \ci \bRd$ be a self-similar set which is the attractor of a self-similar system satisfying OSC. Then \attr is nontrivial if and only if \attr has Minkowski dimension %(or equivalently, Hausdorff dimension)
  strictly less than $d$.
\end{prop}
Unless explicitely stated otherwise, all self-similar sets considered here are assumed to be nontrivial, and the discussion of a self-similar tiling \tiling implicitly assumes that the corresponding attractor \attr is nontrivial and that the corresponding system satisfies OSC.

Denote the set of all finite \emph{words} formed by the alphabet $\{1,\dots,N\}$ by
\linenopax
\begin{equation}\label{eqn:def:words}
  \Words \coloneqq \bigcup_{k=0}^\iy \{1,\dots,N\}^k\,.
\end{equation}
For any word $w=w_1 w_2\dots w_n \in \Words$, let $r_w \coloneqq r_{w_1}\cdot\ldots\cdot r_{w_n}$ and $\simt_w \coloneqq \simt_{w_1} \circ \dots \circ \simt_{w_n}$. In particular, if $w \in \Words $ is the \emph{empty word}, then $r_w=1$ and $\simt_w=\mathrm{Id}$.

\begin{defn}\label{def:self-similar-tiling}%(Self-similar tiling)
  Let $O$ be a feasible open set for $\simtset$.
  %The \emph{geno} associated with $\{S_1,\ldots,S_N\}$ and $O$ %of the tiling is the set $\Gamma\coloneqq O \setminus \simS O$.
  %$G_q, q \in Q$, where we assume $Q$ is finite. The sets $G_q$ are called the \emph{generators} of the tiling.
  The \emph{self-similar tiling} $\sT(O)$ associated with the IFS $\{\simt_1,\ldots,\simt_N\}$
  is the collection of open sets
  \linenopax
  \begin{equation}\label{eqn:def:self-similar-tiling}
    \sT(O) \coloneqq \{ \simt_w(G) \suth w \in \Words\},
  \end{equation}
  where the open set $G \coloneqq O \setminus \cj{\mathbf S O}$ is called the \emph{generator} of the tiling. 
  We call the tiling $\widetilde{\sT}(O) \coloneqq \{ \simt_w(\gG) \suth w \in \Words\}$ generated by 
	\linenopax
	\begin{align}\label{eqn:Gamma}
		\gG=\gG(O) \coloneqq O \setminus \simS O,
	\end{align}
 a \emph{self-similar tiling in a generalized sense}, with the tiles not necessarily being open.
  %due to the fact that Lebesgue measure is not stable under closure (even for open sets!), 
\end{defn}
%We order the words $w^{(1)}, w^{(2)}, \ldots$ of \Words in such a way that the sequence $\sL=\{\scale_j\}_{j = 1}^\iy$ given by $\scale_j := r_{w^{(j)}}$, $j=1,2,\ldots$, is nonincreasing.%
   %In this context, the mapping $\simm_j$ appearing in Def.~\ref{def:fractal-spray} corresponds to $\simt_{w^{(j)}}$.

\begin{remark}%[A note on terminology]%\label{def:}
	Self-similar tilings
	%\linenopax
	%\begin{align}\label{eqn:self-similar-tiling}
	%	\sT(O) \coloneqq \{S_w(G)\suth w\in \Words\},
	%	\qq\text{where}\qq
	%	G \coloneqq O\setminus \cj{\mathbf S O},
	%\end{align}
	% is an open set which is called the \emph{generator} of the self-similar tiling
	generated by $G$ were introduced in \cite{Pe2, SST, GeometryOfSST} and further studied in \cite{TFSST, TFCD, Pointwise, Minko, Demir, Kocak2014, Kocak1, Kocak2}. % (see Rem.~\ref{rmk:pluriphase:generator}).
	The nomenclature stems from the fact (proved in \cite[Thm.~5.7]{GeometryOfSST}) that $\tiling(O)$ is an \emph{open tiling} of $O$ in the sense that
	\linenopax
	\begin{align}\label{eqn:open-tiling}
	\cj{O}= \cj{\bigcup\nolimits_{w \in \Words} \simt_w(G)}\,,
	\end{align}
	where the \emph{tiles} $\simt_w(G)$ are pairwise disjoint open sets.
	%This clarifies that a self-similar tiling is just a specially constructed fractal spray.
 %(With more than one generator, it is, in fact, a collection of fractal sprays, each with the same fractal string $\sL = \{\scale_j\}_{j=1}^\iy$ and a different generator $\gen_q$, $q \in Q$. It may also be viewed as a fractal spray generated on the union set $\bigcup_{q\in Q} G_q$, as the connectedness of the generator is not a requirement for fractal sprays.)
	
%	The terminology ``generalized self-similar tiling'' stems from the observations that
%	\linenopax
%	\begin{align}%\label{eqn:}
%		O = \bigcup_{w\in\Words}S_{w}(\Gamma),
%	\end{align}
%	where the tiles $S_w(\Gamma)$ in $\sT(O)$ are pairwise disjoint sets, and the adjective ``generalized'' is added to avoid confusion with \eqref{eqn:self-similar-tiling}.
\end{remark}
 
\begin{figure}%[b]
  \centering
  \scalebox{0.7}{\includegraphics{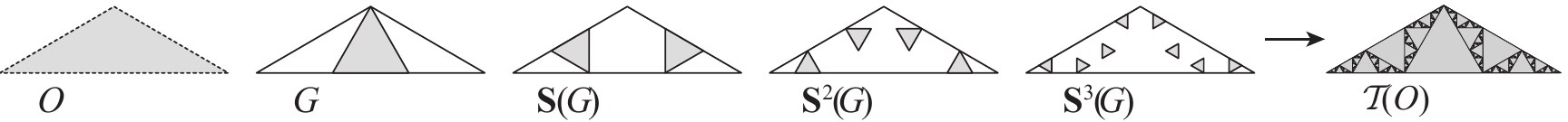}} \\
  \scalebox{0.80}{\includegraphics{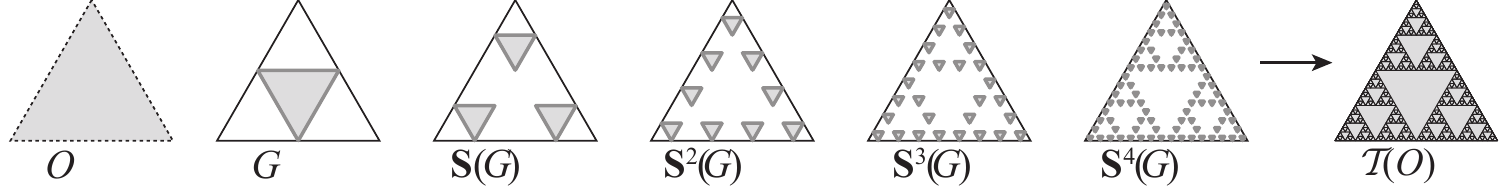}} \vstr[12]\\
  \scalebox{0.82}{\includegraphics{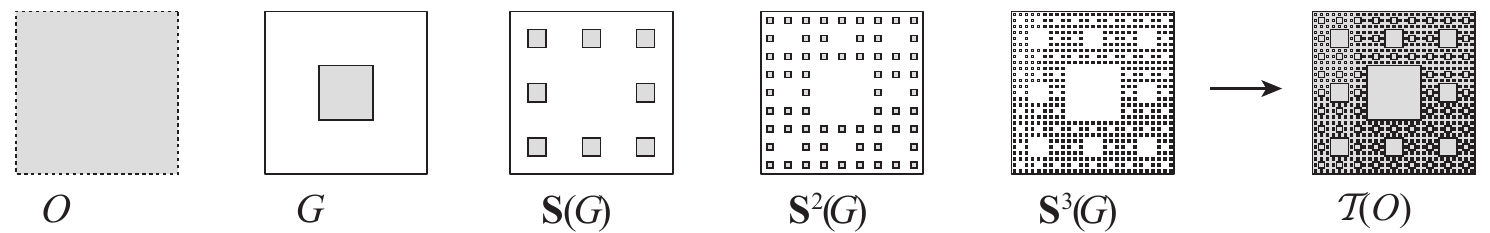}} \vstr[15]
  \caption{\captionsize From top to bottom: a Koch curve tiling, a Sierpinski gasket tiling, and a Sierpinski carpet tiling. In each of these examples, the set $O$ is the interior of the convex hull of \attr, and the set $F$ is monophase w.r.t. $\Gamma$ (see Def.~\ref{def:pluriphase}). 
  The Koch curve tiling does not satisfy the compatibility criterion $\bd O \ci F$ but the other two examples do.}
  \label{fig:examples}
\end{figure}

 We will find it more useful to work in terms of the set $\gG$ instead of $G$ for most of the sequel.
 Observe that $G\subseteq\gG$ and that $\Gamma\setminus G\subset \bigcup_i \bd S_i O$. If $F$ is assumed to be nontrivial, then the set \gG has nonempty interior and we let
	\linenopax
 \begin{equation}\label{eqn:g}
		g \coloneqq \sup\{d(x,F)\suth x \in \gG\}=\sup\{d(x,F)\suth x \in G\}
	\end{equation}
	denote the maximal distance of a point in $\gG$ to $F$. 
	The reason for the use of $\gG$ is that Lebesgue measure is not stable with respect to the closure operation: one may have $\gl_d(U) < \gl_d(\cj{U})$ for an open set $U \ci \bRd$.
	We  remark that
	\linenopax
	\begin{equation}\label{eq:Sunion}
    O = \gG \cup \mathbf{S} O
    = \Gamma \cup \mathbf{S} \Gamma \cup \mathbf{S}^2 O 
= \dots 
%= \Gamma \cup \mathbf{S} \Gamma \cup \dots \cup \mathbf{S}^n \Gamma \cup \mathbf{S}^{n+1} O 
= \bigcup_{k=0}^n \mathbf{S}^k \Gamma \cup \mathbf{S}^{n+1} O,    %
	\end{equation}
where all the unions are disjoint, and hence \eqref{eqn:open-tiling} implies
\linenopax
\begin{align}\label{eq:Tunion}
	\cj{O} = \cj{\bigcup_{w\in \Words} S_w \gG}.
\end{align}
Therefore, $\widetilde{\sT}(O)$ from Def.~\ref{def:self-similar-tiling} gives a tiling of $O$, where the tiles $S_w\gG$ are pairewise disjoint but not necessarily open, justifying the term self-similar tiling in a generalized sense.
Also, \eqref{eq:Sunion} allows for the following nice decomposition of the $\ge$-parallel volume of the attractor which is used in the proof of Thm.~\ref{thm:subresult}:
	\linenopax
	\begin{equation}\label{eq:Tdecomp}
    \lambda_d(F_{\ge})=\sum_{i=1}^N\lambda_d(F_{\ge}\cap S_i O)+\lambda_d(F_{\ge}\cap \Gamma)+\lambda_d(F_{\ge}\setminus O).
	\end{equation}
 This representation is particularly useful for sets $O$ satisfying the projection condition.
%\end{remark}

%----COMMENT------
%----------------
\begin{comment}
We make the obvious yet nonetheless useful remark that
\linenopax
\begin{equation}\label{eq:Tunion}
	O = \Gamma\cup\bigcup\nolimits_{i=1}^N S_i O,
\end{equation}
where all unions are disjoint. This allows for the following nice decomposition of the $\ge$-parallel volume of the attractor:
\linenopax
\begin{equation}\label{eq:Tdecomp}
	\lambda_d(F_{\ge})
	= \lambda_d(F_{\ge}\cap \Gamma)
		+ \sum_{i=1}^N\lambda_d(F_{\ge}\cap S_i O)
		+ \lambda_d(F_{\ge}\setminus O).
\end{equation}
\end{comment}
%----------------

%This representation is particularly useful if the pluriphase and the projection condition are satisfied, which are defined as follows.

\begin{defn}\label{def:metproj}\label{def:projection}
     For a compact set $A\ci\mathbb R^d$, we let $\pi_A$ denote the \emph{metric projection} onto $A$. It is defined on the set of points $x\in\mathbb R^d$ which have a unique nearest neighbour $y$ in $A$ by $\pi_A(x)=y$. Let $O$ be a feasible open set of the self-similar system $\{S_1,\ldots, S_N\}$ with attractor $F$. Then $O$ is said to satisfy the \emph{projection condition} if
\linenopax
\begin{align}\label{eq:PC}
     S_iO\ci\overline{\pi_F^{-1}(S_iF)}\quad\text{for}\ i=1,\ldots,N.
     \end{align}
\end{defn}
If the projection condition is satisfied then
\linenopax
\begin{equation}\label{eq:proj}
   F_{\ge}\cap S_iO=(S_iF)_{\ge}\cap S_iO
   \end{equation}
for each $\ge>0$ and $i=1,\ldots,N$ (see \cite[Lem.~3.19]{MinkContGenForm}).

\begin{comment}\label{rem:1-is-enough}
   For self-similar tilings with more than one generator, one can consider each generator independently, and a tube formula of the whole tiling is then given by the sum of the expressions derived for each single generator. Thus, there is no loss of generality in considering only the case of a single generator, which we will denote by $G$ in the sequel. See, however, Rem.~\ref{rem:lattice-detail} and Rem.~\ref{rem:nonlattice=nonlattice} for further discussion of this issue.
\end{comment}

\begin{remark}\label{rmk:centralopen}
	The \emph{central open set} is a particular choice of feasible open set that exists for any IFS satisfying the OSC; it is defined and studied in \cite{BandtNguyenRao}. Because of its definition in terms of metric projections (see \cite{BandtNguyenRao}), it is easy to see that the central open set will always satisfy the projection condition. It is also clear that $F$ is contained in the central open set, so the \emph{strong} condition is also automatically satisfied. Therefore, it is always possible to find a strong feasible open set which satisfies the projection condition, as long as the OSC holds. A proof of these facts is given in \cite[Prop.~3.17]{MinkContGenForm}.
\end{remark}

\begin{defn}\label{def:pluriphase}
	For a given IFS and a fixed feasible open set $O$, we call the attractor $F$ \emph{pluriphase with respect to the set}  $\Gamma=\Gamma(O)$ (as defined in \eqref{eqn:Gamma}) if and only if there exists a finite partition of the interval $(0,\infty)$ with partition points $0\eqdef a_0<a_1<\dots<a_{M-1}<a_M\defeq g$ such that, for $\ge >0$,
\linenopax
\begin{align}\label{eqn:pluriphase}
	\gl_d(F_\ge \cap \gG) = \sum_{m=1}^M \one_{(a_{m-1},a_m]}(\ge) \sum_{k=0}^d \gk_{m,k} \ge^{d-k}
		+ \one_{(g,\iy)}(\ge) \gl_d(\gG),
\end{align}
for some constants $\gk_{m,k} \in \bR$, where $\one$ denotes a characteristic function and $g$ is as in \eqref{eqn:g}. 
We assume that the representation in \eqref{eqn:pluriphase} is given with $M$ minimal, so that 
%Note that if $F$ is pluriphase w.r.t.\ $\Gamma$, then there exists a partition with $M$ minimal. Without loss of generality, we may therefore impose the following further condition: 
for each $m=1,\dots,M$, there exists a $k \in \{0,\dots,d\}$ with $\gk_{m,k} \neq \gk_{m-1,k}$.
Imposing minimality of $M$, we call $F$ \emph{monophase with respect to} $\Gamma$ if and only if $M=1$ in the above representation.
\end{defn}
	
\begin{remark}\label{rem:pluriphase} 
At the time of writing, there is no known characterization of the pluriphase or monophase conditions in terms of the self-similar system $\{\simt_i\}_{i=1}^N$. However, it is known from \cite{Kocak2} that a convex polytope in \bRd is monophase (with Steiner-like function of class $C^{d-1}$) iff it admits an inscribed $d$-dimensional Euclidean ball (i.e., a $d$-ball tangent to each facet). This includes regular polygons in \bRt and regular polyhedra in \bRd, as well as all triangles and higher-dimensional simplices. Furthermore, it was recently shown in \cite{Kocak2} that (under mild conditions), any convex polyhedron in \bRd ($d \geq 1$) is pluriphase, thereby resolving in the affirmative a conjecture made in \cite{TFSST,TFCD,Pointwise}.
%Recall from \cite{TFCD,Pointwise} that a set is said to be \emph{pluriphase} iff it admits a Steiner-like representation which is piecewise polynomial, i.e., that $(0,\genir)$ can be partitioned into finitely many intervals with $\crv_k(\gen,\ge)$ constant on each interval.
 We refer to \cite{Kocak2} for further relevant interesting results.
\end{remark}

\begin{remark}\label{rmk:pluriphase:generator}
  In \cite{TFSST,TFCD} the notions monophase, pluriphase and the symbol ``$g$'' were introduced for the generator of a self-similar tiling satisfying the compatibility condition $\bd O \ci \attr$.   
	The reader should be aware that these terms in the present paper only match previous usage in the literature for the case when $\bd O \ci \attr$.
	In this context, the upper endpoint $g$ of the relevant interval in Def.~\ref{def:pluriphase} was defined as the \emph{inradius} $\tilde{g}$ of $G$, i.e., the maximal radius of an open metric ball contained in the set $G$.
        Moreover, the generator $G$ (as a set) was called monophase if $\lambda_d(G_{-\ge})$ is polynomial in $\ge$ for $\ge\in(0,\tilde{g})$, where $G_{-\ge}\defeq\{x\in G\suth d(x,G^c)\leq\ge\}$ and pluriphase if $\lambda_d(G_{-\ge})$ is piecewise polynomial in $\ge$ for $\ge\in(0,\tilde{g})$.
 However, Fig.~\ref{fig:examplesnew} shows an example where $g \neq \tilde{g}$ and where $G$ (as a set) is monophase but $\attr$ is not monophase w.r.t.\ $\gG$. We return to this example in Sec.~\ref{sec:examples}, Ex.~\ref{ex:Sierpinski:gasket}.
It is clear that the inradius and the notions mono- and pluriphase for sets are not the proper concept for the situation where $\bd O \not \ci F$. It is also clear from this observation that
	%one is required to consider the pair $(F,O)$ instead of just the set $G$, and so
	\eqref{eqn:g} and Def.~\ref{def:pluriphase} are the natural extensions to the present (more general) setting.
\end{remark}

\begin{remark}\label{def:monophase-examples}
	Some examples of self-similar tilings associated to familiar fractal sets are shown in Fig.~\ref{fig:examples}. In each case, there is a connected monophase generator.
In Fig.~\ref{fig:examplesnew} an alternative tiling associated with the Sierpinski gasket is provided. Here, the generator is not connected.
\end{remark}

\begin{figure}%[b]
  \centering
  \scalebox{0.48}{\includegraphics{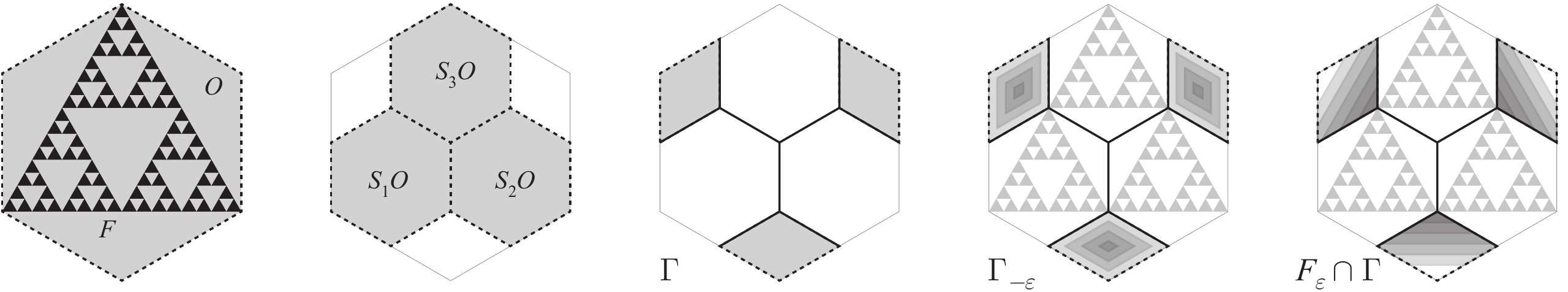}}
  \caption{\captionsize A Sierpinski gasket tiling alternative to the one from Fig.~\ref{fig:examples}. Here, $O$ is not the interior of the convex hull of $F$, but rather the \emph{central open set} discussed in \cite{BandtNguyenRao}. The set $F$ is pluriphase (but not monophase) w.r.t. $\Gamma$, while the set $\gG$ (as a set) is monophase. At right, the sets $\gG_{-\ge} \coloneqq (\bd \gG)_\ge \cap \gG$ and $F_{\ge} \cap \gG$ are shown for several values of \ge. For this example,  $g=(2\sqrt3)^{-1}$ and $\tilde{g}=1/8$; see Rem.~\ref{rmk:pluriphase:generator}. This example will be further studied in Sec.~\ref{sec:examples}, Ex.~\ref{ex:Sierpinski:gasket}.}
  \label{fig:examplesnew}
\end{figure}

\subsection{Lattice, nonlattice and the renewal theorem}

\begin{defn}\label{def:lattice}
	Consider a family of similarity mappings $S=\{\simt_1,\ldots,\simt_N\}$, and let $r_i$ denote the scaling ratio of $S_i$. The family is said to be of \emph{lattice type} iff there is an $r>0$ such that each scaling ratio $r_i$ can be written as $r_i = r^{k_i}$ for some integer $k_i$, and to be of \emph{nonlattice type} otherwise. There is a smallest number $r>0$ for which the aforementioned representation can be found. We always use this minimal $r$, and we say that $S$ is \emph{lattice with base $r$}.
	%The associated self-similar set, or an associated self-similar tiling, is correspondingly also referred to as \emph{lattice} or \emph{nonlattice}.
\end{defn}

An extended discussion of the implications of the lattice/nonlattice dichotomy may be found in \cite[Thm.~3.6]{FGCD}. The lattice/nonlattice dichotomy also appears in probabilistic renewal theory, where the usual nomenclature is ``arithmetic/nonarithmetic''. For more details, see \cite[\S{XIII}]{Feller}, \cite[\S7]{Falconer_Techniques} or \cite[\S4]{Steffen:thesis}.

For use in the sequel, we include here a version of the renewal theorem formulated for a discrete probability distribution $\sum_{i=1}^N p_i \gd_{y_i}$, where $\gd_y$ is a point mass (Dirac measure) concentrated at $y \in \bR$. This theorem will be applied to the distribution
\linenopax
\begin{align}\label{eqn:Moran-probability}
	\sum_{i=1}^N r_i^D \gd_{y_i},
\end{align}
where $D$ is the \emph{similarity dimension} of $F$. The similarity dimension is the unique positive real number \ga that satisfies the Moran equation $r_1^\ga + r_2^\ga + \dots + r_N^\ga = 1$, i.e., the unique $D>0$ that makes \eqref{eqn:Moran-probability} into a discrete probability distribution.

\begin{theorem}[Renewal Theorem (see {\cite[Cor.~7.3]{Falconer_Techniques}} or {\cite[\S4]{Steffen:thesis}}]\label{thm:renewal}%, {\cite[\S4]{Steffen:thesis}})]
	Let $p_1,\dots,p_N \in (0,1)$ satisfy $\sum_{i=1}^N p_i = 1$, and let $y_1, \ldots, y_N > 0$. Let $z:\bR\to\bR$ be a function with a discrete set of discontinuities which satisfies
	\linenopax
	\begin{align}\label{eqn:zbounded}
		|z(t)| \leq c_1 \ee^{-c_2|t|},
		\qq\text{for all } t \in \bR,
	\end{align}
	for some constants $0 < c_1, c_2 < \iy$. Also, let $Z:\bR\to\bR$ be the unique solution of the renewal equation
	\linenopax
	\begin{align}%\label{eqn:}
		Z(t) = z(t) + \sum_{i=1}^N p_i Z(t-y_i)
	\end{align}
%\marginpar{Is the very last condition (... bounded on the half line ...) really needed?}
       which satisfies $\lim_{t \to -\iy} Z(t) = 0$. % and which is bounded on the half-line $(-\infty,a]$ for every $a\in\mathbb R$. 
       Then the following holds:
	\begin{enumerate}
        \item\label{it:rt:lattice} If $\{y_1,\ldots, y_N\}\ci h\cdot\mathbb Z$ and $h>0$ is maximal as such, then%In the lattice case, with common base $r$ and $h:=-\ln r$,
		\linenopax
		\begin{align}\label{eqn:Z(t)-sim-sum}
			Z(t) \sim \frac h\gh \sum_{\ell\in\mathbb Z} z(t-\ell h), \qquad \text{ as } t\to\infty.
		\end{align}
	\item\label{it:rt:nonlattice} If there does not exist $h>0$ such that $\{y_1,\ldots,y_N\}\ci h\cdot\mathbb Z$, then %In the nonlattice case,
		\linenopax
		\begin{align}%\label{eqn:}
			\lim_{t \to \iy} Z(t) = \frac1\gh \int_{-\iy}^\iy z(\gt) \dgt.
		\end{align}
	\end{enumerate}
        Here, $\gh\coloneqq\sum_{i=1}^N y_ip_i$.
	Moreover, %$Z$ is uniformly bounded in \bR, and
        in both %either the lattice or nonlattice
        cases, we have
	\linenopax
	\begin{align}%\label{eqn:}
		\lim_{T \to \iy} \frac1T \int_0^T Z(t) \dt = \frac1\gh \int_{-\iy}^\iy z(t) \dt.
	\end{align}
	%\marginpar{\Sar{Do we need this? If so, should include \cite[\S4]{Steffen:thesis} as a reference.}}
\end{theorem}
In \eqref{eqn:Z(t)-sim-sum}, the notation $g\sim f$, as $t\to\infty$, means that $g$ is \emph{asymptotic to} $f$ as $t\to\infty$ in the sense that for any $\delta > 0$, there is a number $s=s(\delta)$ for which
\linenopax
\begin{align}\label{eqn:asymptotic}
	(1-\delta) f(t) \leq g(t) \leq (1+\delta)f(t),
	\qq\text{for all } t \geq s.
\end{align}
\begin{remark}
    If the self-similar system $\{S_1,\ldots,S_N\}$ is lattice, then there exist $r>0$ and $k_i\in\mathbb N$ such that $r_i=r^{k_i}$, where $r_i$ denotes the scaling ratio of $S_i$ for $i=1,\ldots,N$. In this case $\{-\ln r_1,\ldots,-\ln r_N\}=\{-k_1\ln r,\ldots,-k_N\ln r\} \ci -\ln r\cdot\mathbb Z$ and thus, we are in case \ref{it:rt:lattice} of the renewal theorem. On the other hand, if $\{S_1,\ldots,S_N\}$ is nonlattice, then we are in case \ref{it:rt:nonlattice} of the renewal theorem.
\end{remark}

\begin{remark}[A brief dictionary]%\label{def:}
	The renewal theorem above is given in terms of the additive variable $t \in \bR$ but will be applied in the context of the multiplicative variable $\ge \in (0,g]$. For the reader's convenience, we offer the following translation of symbols corresponding to the change of variables $\ge = \ee^{-t}$: \\
	\begin{center}
	\begin{tabular}{c|c|c|c|c|c|c}
		$\ee^{-t}$ & $t \to \iy$  & $-\ln g \leq t < \iy$ & $\ee^{-h}$ & $t - \ell h$     &  $\frac1h (t-\ln g)$  & $p_i$\\ \hline
		\ge        & $\ge \to 0$  & $0 < \ge \leq g$      & $r$        & $-\ln(r^{-\ell}\ge)$  & $-\log_r (g^{-1}\ge)$ & $r_i^D$
	\end{tabular}
	\end{center}
\end{remark}

%\marginpar{I removed Gatzouras' Theorem here. If we want it in then I think that we might also want to present the existing lattice results.}
%\begin{theorem}[Gatzouras' Theorem]
%	\label{thm:Gatzouras}
%	Let $F \ci \bR^d$ be a self-similar set satisfying OSC, and let $D$ be the similarity dimension of $F$. Then the average Minkowski content of $F$ (see Def.~\ref{def:Minkowski-measurable}) exists and coincides with the strictly positive value
%	\linenopax
%	\begin{align*}%\label{eqn:}
%		X_d \coloneqq \frac1\gh \int_0^1 \ge^{D-d-1} R_d(\ge) \de,
%	\end{align*}
%	where the function $R_d:(0,\iy) \to \bR$ is given by
%	\linenopax
%	\begin{align*}%\label{eqn:}
%		R_d(\ge) = \gl_d(F_\ge) - \sum_{n=1}^N \one_{(0,r_n]}(\ge) \gl_d((S_n F)_\ge)
%	\end{align*}
%	and $\gh := -\sum_{n=1}^N r_N^D \ln r_n$.	
%	If $F$ is nonlattice, then also the Minkowski content $\Mink^D(F)$ of $F$ exists and equals $X_d$.
%\end{theorem}

%%!TEX root = lattice-nonmeasurability_v13.tex

\section{Statement and proof of the main results}\label{sec:results}

%\begin{theorem}\label{thm:main-result}
%	\marginpar{So far without proof :-(}
%  Let $F$ be the attractor of a self-similar system \simtset which satisfies the open set condition with some strong feasible open set $O$.
%  Assume that $O$ satisfies the projection condition and that $(F,O)$ satisfies the pluriphase condition. If $F$ is nontrivial and \simtset is lattice, then $F$ is not Minkowski measurable.
%\end{theorem}

In order to prove Thm.~\ref{thm:main-result}, we first prove a theorem which provides information on the asymptotic behavior of the parallel volume of the self-similar set $F\subseteq\mathbb R^d$ under weaker conditions. This intermediate result is of independent interest as it does not require the pluriphase condition to be satisfied, and thus may provide an avenue for eventually removing this hypothesis. 
%This result is formulated in terms of the multiplicatively periodic
%function $p:(0,g] \to \bR$ of period $r>0$ defined by
%\linenopax
%\begin{align}\label{eqn:p(e)}
%	p(\ge) \coloneqq
%	%\begin{cases}
%		\ge^{D-d}\sum_{\ell \in\mathbb Z} r^{\ell(D-d)} \gl_d(F_{r^{\ell}\ge} \cap \gG),
%		\qq\text{for } \ge > 0. %&rg < \ge \leq g,\\
%		%p(r^{-1}\ge), & 0 < \ge \leq  rg.\\
%		%\end{cases}
%\end{align}
%%and the periodicity relation $p(\ge)=p(r^{-1}\ge)$ for $\ge \leq rg$.
%The reader can easily check the periodicity identity $p(r\ge)=p(\ge)$. Recall from \eqref{eqn:g} that $g = \sup\{d(x,F)\suth x \in \gG\}$ and from \eqref{eqn:Gamma} that $\gG \coloneqq O \setminus \simS O$.
%%In other words, $p(\ge)$ is given on $(rg,g]$ and then extended to $(0,g]$ by periodicity, so that $p(\ge) = p(r\ge)$ on this interval.
%We also continue to denote the contraction ratio of $S_i$  by $r_i$ for $i=1,\ldots,N$. 
In analogy to \eqref{eqn:asymptotic} we say that $f\sim g$ as $\ge\to 0$ iff $f\circ h\sim g\circ h$ as $t\to\infty$, where $h(t)\defeq \exp(-t)$. If $f,g:(0,\infty)\to(0,\infty)$ this is equivalent to assuming that $\lim_{\ge \to 0} f(\ge)/g(\ge) = 1$.

\begin{theorem}\label{thm:subresult}
	 Let $F\subset\bR^d$ be the attractor of a self-similar system $S=\simtset$ satisfying the OSC and let $r_i$ denote the contraction ratio of $S_i$ for $i=1,\ldots,N$. Assume $F$ is nontrivial (i.e.~$D\defeq\textup{dim}_{\mathcal M}(\attr)<d$). Let $O$ be an arbitrary strong feasible set satisfying the projection condition, $\gG\defeq O\setminus \mathbf{S}O$ and $g$ as in \eqref{eqn:g}.
	 If $S$ is lattice with base $r$ %, i.e.\ $r>0$ is the smallest positive number such that each $r_i$ may be written as $r_i=r^{k_i}$ with $k_1,\ldots,k_N\in\mathbb N$, 
then
  %Let $\Gamma \coloneqq O\setminus \simS O$ denote the geno associated with $\{S_1,\ldots,S_N\}$ and $O$.
  \linenopax
  \begin{align}\label{eq:asympeins}
  	\ge^{D-d}\gl_d(F_{\ge})
  	&\sim \frac{\ln r}{\sum_{i=1}^N r_i^D\ln r_i} p(\ge),
	\qq\text{as } \ge \to 0,
  \end{align}
  where $p:(0,g] \to \bR$ is defined by
  \linenopax
  \begin{align}\label{eqn:p(e)}
    	p(\ge) \coloneqq
        \ge^{D-d}\sum_{\ell \in\mathbb Z} r^{\ell(D-d)} \gl_d(F_{r^{\ell}\ge} \cap \gG),
	\qq\text{for } \ge > 0. 
  \end{align}
  Moreover, for $\ge \in (rg,g]$, $p$ has the alternative representation
  \linenopax
  \begin{align}\label{eq:asympsum}
  	p(\ge)
  	= \ge^{D-d} \left[ \frac{\gl_d(\gG)}{r^{D-d}-1} + \sum_{\ell = 0}^{\infty} r^{\ell(D-d)} \gl_d(F_{r^{\ell}\ge}\cap \gG) \right].
  \end{align}
\end{theorem}

If one can show that the periodic function $p$ is non-constant, then Thm.~\ref{thm:subresult} implies that $\limsup_{\ge\to 0}\ge^{D-d}\gl_d(F_{\ge})> \liminf_{\ge\to 0}\ge^{D-d}\gl_d(F_{\ge})$ and hence that $F$ is not Minkowski measurable. On the other hand, if the function $p$ is constant, then \eqref{eq:asympeins} implies immediately that $F$ is Minkowski measurable. Note that in both cases $p$ is a strictly positive function. This is for instance obvious from \eqref{eq:asympsum} since the first term is strictly positive and all terms in the second summation are non-negative.
We save these important observations for later use:

\begin{cor} \label{cor:subresult}
  Under the hypothesis of Thm.~\ref{thm:subresult}, the self-similar set $F$ is Minkowski measurable if and only if the function $p$ (given by \eqref{eqn:p(e)} or \eqref{eq:asympsum}) is constant, that is, $p(\ge)=C$ for some constant $C>0$ and all $\ge>0$.
Moreover, in this case the $D$-dimensional Minkowski content of $F$ is given by
\linenopax
\[\sM_D(F)=\frac{\ln r}{\sum_{i=1}^N r_i^D\ln r_i}\cdot C.\]
\end{cor}

%This is exactly how we will proceed when proving Thm.~\ref{thm:main-result}.
Note that Thm.~\ref{thm:subresult} and Cor.~\ref{cor:subresult} apply to all nontrivial self-similar sets satisfying OSC. There is no monophase or pluriphase condition present and the projection condition on its own does not impose any restrictions. There exists always a strong feasible set $O$ for which it is satisfied, see Rem.~\ref{rmk:centralopen}. Only trivial self-similar sets are excluded. In this case, the statement of Thm.~\ref{thm:subresult} does not make sense, since $\Gamma=\emptyset$ and thus $p\equiv 0$. But such sets are always Minkowski measurable and there is no need for a statement like this. (Note that for a trivial self-similar set $F\subset\bR^d$, the $d$-dimensinoal Minkowski content is given by $\sM_d(F)=\lambda_d(F)$.)

\begin{proof}[Proof of Thm.~\ref{thm:subresult}]
	We decompose the parallel volume of $F$ through
  \linenopax
		\begin{equation}\label{eq:decompFe}
			\gl_d(F_{\ge})
			= \gl_d(F_{\ge}\setminus O)+\gl_d(F_{\ge}\cap O).
		\end{equation}
	For the first summand on the right hand side of \eqref{eq:decompFe}, we note that $\simS O\ci O$ so that $O^c\ci\simS O^c\ci(\simS O^c)_{\ge}$. By \cite[Cor.~5.6.3]{Steffen:thesis} we know that there exist $c,\gamma>0$ such that
	\linenopax
	\begin{equation}\label{eq:FeoutsideSO}
	\gl_d(F_{\ge}\cap(\simS O^c)_{\ge})\leq c\ge^{d-D+\gamma},
	\qq\text{for } \ge\in(0,1).
	\end{equation}
	Note that it is this estimate which requires the hypothesis that $O$ is a strong feasible set.
	Equation \eqref{eq:FeoutsideSO} implies $\gl_d(F_{\ge}\setminus O)\leq c\ge^{d-D+\gamma}$, whence
	\linenopax
	\begin{equation}\label{eq:outsideO}
		\lim_{\ge\to 0}\ge^{D-d}\gl_d(F_{\ge}\setminus O)=0.
	\end{equation}
	Now we turn to the more interesting second summand on the right hand side of \eqref{eq:decompFe}.
	From \eqref{eq:Tdecomp}, we have
	\linenopax
	\begin{equation}\label{eq:voldecomposed}
		\gl_d(F_{\ge}\cap O)
		= \sum_{i=1}^N\gl_d(F_{\ge}\cap S_i O) + \gl_d(F_{\ge}\cap \gG),
		\qq\text{for } \ge>0.
	\end{equation}
	We deduce from \eqref{eq:proj} that
	\linenopax
	\begin{equation}\label{eq:scaling}
		\gl_d(F_{\ge}\cap S_i O)
		= \gl_d((S_iF)_{\ge}\cap S_i O)
		= r_i^d\cdot\gl_d(F_{\ge/r_i}\cap O).
	\end{equation}
	Note that it is \eqref{eq:proj} which uses the hypothesis concerning the projection condition.
	We multiply \eqref{eq:voldecomposed} by $\ge^{D-d}\one_{(0,g]}(\ge)$ to obtain
	\linenopax
	\begin{align*}
		\ge^{D-d}\one_{(0,g]}(\ge)\gl_d(F_{\ge}\cap O)
		&=\sum_{i=1}^N r_i^D\left(\ge/r_i\right)^{D-d}\one_{(0,g]}\left(\ge/r_i\right)\gl_d(F_{\ge/r_i}\cap O)\\
		&\quad +\sum_{i=1}^N r_i^d\ge^{D-d}\one_{(r_ig,g]}(\ge)\gl_d(F_{\ge/r_i}\cap O)
		+\ge^{D-d}\one_{(0,g]}(\ge)\gl_d(F_{\ge}\cap \gG).
	\end{align*}
	Setting
	\linenopax
	\begin{align}\label{eqn:tildeZ}
	\tilde{Z}(\ge)\defeq \ge^{D-d}\one_{(0,g]}(\ge)\gl_d(F_{\ge}\cap O),
	\end{align}
	the previous equation can be rewritten as
	\linenopax
	\begin{align*}
		\tilde{Z}(\ge)
		=\sum_{i=1}^N r_i^D\tilde{Z}(\ge/r_i) + \tilde{z}(\ge)
	\end{align*}
	where we have introduced
	\linenopax
	\begin{align}\label{eqn:z(e)}
		\tilde{z}(\ge)
		\coloneqq \sum_{i=1}^N r_i^d\ge^{D-d}\one_{(r_ig,g]}(\ge)\gl_d(F_{\ge/r_i}\cap O)
		+ \ge^{D-d}\one_{(0,g]}(\ge)\gl_d(F_{\ge}\cap \gG)
	\end{align}
	%From the open tiling \eqref{eqn:open-tiling}, we moreover obtain the representation $O=\bigcup_{w\in \Words} S_w\Gamma$ which likewise is a disjoint union.
	The definition of $g$ yields $\gG \ci F_g$, which implies for all $w\in \Words$ that $S_w \gG \ci S_w(F_g) \ci (S_w F)_g \ci F_g$. Consequently, \eqref{eq:Tunion} yields $O \ci \cj{O} \ci F_g$ and we have $F_{\ge} \cap O = O$ for any $\ge \geq g$. Since $\ge/r_i>g$ for $\ge\in(r_i g,g]$, we can thus reduce $\tilde{z}(\ge)$ to
	\linenopax
	\begin{equation}\label{eq:z}
		\tilde{z}(\ge)
		= \ge^{D-d}\cdot\left(\gl_d(O)\sum_{i=1}^N r_i^d\one_{(r_ig,g]}(\ge)
			+ \one_{(0,g]}(\ge)\gl_d(F_{\ge}\cap \gG)\right).
	\end{equation}
	In order to be able to apply the renewal theorem (Thm.~\ref{thm:renewal}), we make the variable transformation $\ge=\ee^{-t}$ with $t\in\mathbb R$ and write $Z(t)\defeq \tilde{Z}(\ee^{-t})$ and  $z(t)\defeq \tilde{z}(\ee^{-t})$. This gives
	\linenopax
	\begin{equation}\label{eq:renewalZ}
		Z(t) = \sum_{i=1}^N r_i^D Z(t+\ln r_i) +z(t).
	\end{equation}
	Since $\gG \ci \simS O^c$, the function $z$ satisfies \eqref{eqn:zbounded} by \eqref{eq:FeoutsideSO}. Moreover, $z$ clearly has a finite set of discontinuities.
	A consequence of $\one_{(0,g]}(\ge)$ being a factor of $\tilde{Z}(\ge)$ and thus $\one_{[-\ln g,\infty)}(t)$ being a factor of $Z(t)$ is that $\lim_{t\to-\infty} Z(t)=0$ holds true. % and that $Z$ is bounded on each half-line $(-\infty,a]$.
	By Moran's equation, $\sum_{i=1}^N r_i^D=1$. Thus, Thm.~\ref{thm:renewal} is applicable to $Z$ with $p_i\defeq r_i^D$ and $y_i\defeq-\ln r_i$ for $i=1,\ldots,N$.
	Since the lattice condition gives $r_i=r^{k_i}$ for $i=1,\ldots,N$, we have $\{y_1,\ldots,y_N\}=\{-k_1\ln r,\ldots,-k_N\ln r\}\ci-\ln r\cdot\mathbb Z$. Therefore, Thm.~\ref{thm:renewal}\ref{it:rt:lattice} yields
	\linenopax
	\begin{equation*}
		Z(t)\sim\frac{\ln r}{\sum_{i=1}^N r_i^D\ln r_i}\sum_{\ell\in\mathbb Z} z(t - \ell h)\q\text{as}\ t\to\infty.
	\end{equation*}
	In $\ge$-notation, this is
	\linenopax
	\begin{equation}\label{eq:Zasymp}
		\tilde{Z}(\ge)\sim\frac{\ln r}{\sum_{i=1}^N r_i^D\ln r_i}\sum_{\ell\in\mathbb Z} \tilde{z}(\ge\cdot r^{-\ell})\q \text{as}\ \ge\to 0.
	\end{equation}
	Noting that the above sum is absolutely convergent (all terms are positive) and using \eqref{eq:z}, we have
	\linenopax
	\begin{align}
	\widetilde p(\ge)&\defeq\sum_{\ell\in\mathbb Z} \tilde{z}(\ge\cdot r^{-\ell})\nonumber
	%&=\sum_{\ell\in\mathbb Z}(r^{-\ell} \ge)^{D-d}\cdot\left(\gl_d(O)\sum_{i=1}^N r_i^d\one_{(r_ig,g]}(r^{-\ell} \ge)
	%	+\one_{(0,g]}(r^{-\ell} \ge)\gl_d(F_{r^{-\ell} \ge}\cap \Gamma)\right)\nonumber\\
	=\ge^{D-d}\left[\gl_d(O)\sum_{i=1}^N r_i^d\sum_{\ell\in\mathbb Z} r^{-\ell(D-d)}\one_{(r_ig,g]}(r^{-\ell} \ge)\right. \nonumber \\ &\hstr[25] \left.
		+\sum_{\ell\in\mathbb Z}r^{-\ell(D-d)}\one_{(0,g]}(r^{-\ell} \ge)\gl_d(F_{r^{-\ell} \ge}\cap \gG)\right].\label{eq:sumz}
	\end{align}	
Define
\linenopax
\begin{align}%\label{eqn:}
  L(\ge) \defeq \left\lfloor\log_r \tfrac{\ge}{g}\right\rfloor
  \q\text{and}\q
  L^i(\ge) \defeq \left\lfloor\log_r \tfrac{\ge}{g}-k_i\right\rfloor = L(\ge) - k_i,
\end{align}
%\marginpar{Floor or ceiling function? Or integer part? I think it's floor ... The integer part function agrees with floor for $x \geq 0$ and ceiling for $x \neq 0$, so it must not be right.}
where $\lfloor x\rfloor$ denotes the floor function of $x\in\bR$, so that $x \mapsto \lfloor x\rfloor+1$ rounds $x$ up to the nearest integer strictly larger than $x$. Upon noting that $r_i g < r^{-\ell}\ge \leq g$ iff $\log_r \frac{\ge}{r_i g} < \ell \leq \log_r \frac{\ge}{g}$ iff $ L^i(\ge)+1\leq \ell\leq L(\ge)$ for $\ell \in\mathbb Z$, we see that \eqref{eq:sumz} becomes
\linenopax
\begin{align}\label{eqn:pLseries}
  \widetilde p(\ge)
  &= \ge^{D-d} \left[ \gl_d(O) \sum_{i=1}^N r_i^d \sum_{\ell = L^i(\ge)+1}^{L(\ge)} r^{-\ell(D-d)} + \sum_{\ell = -\iy}^{L(\ge)} r^{-\ell(D-d)} \gl_d(F_{r^{-\ell}\ge}\cap \gG) \right].
\end{align}
However, we know that $\tilde{p}$ is by definition multiplicatively periodic with multiplicative period $r$, so it suffices to work with $\ge \in (rg,g]$, in which case $L(\ge) = 0$ and $L^i(\ge) = -k_i$.
Using the geometric series, the Moran equation $\sum_{i=1}^N r_i^D=1$ and that $r_i=r^{k_i}$, especially the first summand in \eqref{eqn:pLseries} simplifies significantly and we obtain
	\linenopax
\begin{align}\label{eq:p2}
  \widetilde p(\ge)
  &= \ge^{D-d} \left[ \frac{\gl_d(O)}{1-r^{D-d}}\left (\sum_{i=1}^N r_i^d-1\right) + \sum_{\ell = 0}^{\infty} r^{\ell(D-d)} \gl_d(F_{r^{\ell}\ge}\cap \gG) \right],
  \qq\text{for } \ge \in (rg,g].
\end{align}
Note that
	\linenopax
\begin{align*}
  \gl_d(\gG)
    =\gl_d\left(O\setminus \bigcup_{i=1}^N S_iO\right)
    =\gl_d(O)-\sum_{i=1}^N\gl_d(S_i O)
    =\gl_d(O)\left(1-\sum_{i=1}^N r_i^d\right).
\end{align*}
Therefore, \eqref{eq:p2} becomes
	\linenopax
\begin{align}\label{eq:p3}
 \widetilde p(\ge)
  &= \ge^{D-d} \left[ \frac{\gl_d(\gG)}{r^{D-d}-1} + \sum_{\ell = 0}^{\infty} r^{\ell(D-d)} \gl_d(F_{r^{\ell}\ge}\cap \gG) \right],
  \qq\text{for } \ge \in (rg,g], %\\
\end{align}
which shows that the asymptotic relation \eqref{eq:asympeins} in Thm.~\ref{thm:subresult} holds indeed with the (periodic continuation of the) function $p$ given by \eqref{eq:asympsum}.
Using that $\gl_d(F_{r^{\ell}\ge}\cap \gG)=\gl_d(\gG)$ for $\ge\in(rg,g]$ and $\ell\leq-1$ we obtain from the geometric series expansion that
	\linenopax
\begin{align}\label{eq:p4}
  \widetilde p(\ge)
  &=\ge^{D-d} \sum_{\ell\in\mathbb Z} r^{\ell(D-d)} \gl_d(F_{r^{\ell}\ge}\cap \gG),
  \qq\text{for } \ge \in (rg,g].
\end{align}
Thus $\widetilde p(\ge)=p(\ge)$ for $\ge\in (rg,g]$, where $p$ is as defined in \eqref{eqn:p(e)}, and the periodicity of $p$ implies now that the representation \eqref{eqn:p(e)} (resp. \eqref{eq:p4}) is in fact valid for all $\ge>0$.
%If $\ge\in(r^{s+1}g,r^sg]$ with $s\in\mathbb N$, then $r^{-s}\ge\in(r g,g]$. The periodicity of $\widetilde p$ gives $\widetilde p(\ge)=\widetilde p(r^{-s}\ge)$ for which we can use \eqref{eq:asympsum} (resp. \eqref{eq:p4}) and conclude that
This shows that \eqref{eq:asympeins} holds also with $p$ given by \eqref{eqn:p(e)} and concludes the proof of Thm.~\ref{thm:subresult}.
\end{proof}

\begin{remark}\label{rmk:monophaseQ}
	Under the additional assumptions that $\bd O\ci F$ and that $\overline O\setminus \overline{\mathbf S O}$ possesses a finite number of connected components, Thm.~\ref{thm:subresult} was proven in \cite[Thm.~2.38]{Kombrink:thesis}, the Ph.\,D. thesis of the first author. There, the proof builds on results which were shown for the more general class of self-conformal sets by means of renewal theory in symbolic dynamics. Consequently, it contains arguments to overcome difficulties in the conformal setting which do not occur in the self-similar situation. The proof of Thm.~\ref{thm:subresult} presented here is shorter and more direct for the self-similar situation. Moreover, the statement of Thm.~\ref{thm:subresult} is more general; neither does it require $\bd O\ci F$, nor the assumption that  $\overline O\setminus \overline{\mathbf S O}$ possesses a finite number of connected components.
	
%\marginpar{Delete or move this part?} [On the other hand, \cite[Thm.~2.38]{Kombrink:thesis} not only gives the asymptotic behaviour of $\ge^{D-d}\gl_d(F_{\ge})$ as $\ge\to0$, but also states that whenever the connected components of $\overline O\setminus \overline{\mathbf S O}$ are monophase in the sense of Rem.~\ref{rmk:pluriphase:generator}, then $F$ is not Minkowski measurable.]
%	Below, in Cor.~\ref{cor:severalmonophase}, we recover this statement as a corollary to Thm.~\ref{thm:subresult}. %, where the condition $\bd O\ci F$ is not required.
%	Note that  $\bd O\ci F$ implies that $\gG\setminus\inn\Gamma\ci F$ and thus, that $\gl_d(\Gamma\setminus\inn\Gamma)=0$, where $\Gamma$ is the geno associated with the self-similar system and $O$. Thus, the condition $\gl_d(\Gamma\setminus\inn\Gamma)=0$ in Cor.~\ref{cor:severalmonophase} is the natural extension to self-similar systems which do not satisfy $\bd O\ci F$.
\end{remark}

\subsection{The case when $D = \dim_\sM F$ is not an integer.}

Now we restate and prove part \ref{it:intro:noninteger} of Thm.~\ref{thm:main-result}, the case when $D$ is not an integer; the integer case is discussed in Thm.~\ref{thm:pluriphase-integer-result}.

\begin{theorem}\label{thm:pluriphase-result}
	Let $F$ be a self-similar set which is the attractor of the lattice self-similar system $S=\{S_1,\dots,S_N\}$, $N\geq 2$ satisfying the OSC. Assume the Minkowski dimension $D = \dim_\sM F$ is not an integer and that we can find a strong feasible open set $O$ satisfying the projection condition (see Def.~\ref{def:metproj}) such that $F$ is pluriphase with respect to $\Gamma(O)$.
%which satisfies the following additional conditions:
%	\begin{enumerate}[label=(\roman*)]
%		\item (Projection condition;) $S_i O \subseteq \overline{\pi_{F}^{-1}(S_iF)}$ for $i=1,\ldots,N$, where $\pi_F$ denotes the metric projection onto $F$,
%		\item (Pluriphase; see Def.~\ref{def:pluriphase}) \;	 $O$ is pluriphase.
%	\end{enumerate}
	%We also assume that $D = \dim_\sM F$ is not an integer (see Thm.~\ref{thm:pluriphase-integer-result} for the case when $D$ is an integer).
	In this situation, $F$ is not Minkowski measurable.
\end{theorem}

\begin{remark}\label{rem:hypotheses-clarification}
%	Recall from Rem.~\ref{rmk:centralopen} that whenever OSC is satisfied, one can always find a strongly feasible open set satisfying the projection condition, and thus imposing these conditions restricts the possible open sets $O$ which may be considered, but not the collection of fractals (or self-similar systems) to which the theorem applies.
	Note that the hypothesis that $D$ is not an integer implies that $D<d$. Therefore, Prop.~\ref{cor:OSC-dimension-d-implies-trivial} ensures that the nontriviality condition will be met for any feasible open set. Therefore, the set $\Gamma$ is always nonempty and we do not need to include nontriviality in the hypothesis of Thm.~\ref{thm:pluriphase-result}.
\end{remark}

\begin{proof}[Proof of Thm.~\ref{thm:pluriphase-result}.]
We assume the partition we work with is minimal (cf.~Def.~\ref{def:pluriphase}). In the case $M\geq 2$ we will use the fact that for $\ell\in\mathbb Z$ and $m=2,\ldots,M$ we have the equivalences
\linenopax
\begin{align*}
	r^{\ell}\ge\in\left(a_{m-1},a_m\right]\ \
	\Leftrightarrow\ \ \ell\in \left[\log_r\tfrac{a_m}{\ge},\log_r\tfrac{a_{m-1}}{\ge}\right)\cap\mathbb Z\ \
	\Leftrightarrow\ \ \ell\in\Big{\{}L_m(\ge),\ldots,L_{m-1}(\ge)-1\Big{\}}.
\end{align*}
Here
\linenopax
\begin{align}\label{eqn:Lm}
	L_m(\ge) \coloneqq \left\lceil \log_r \frac{a_m}\ge \right\rceil, \qquad m=1,\ldots, M\,,
\end{align}
%\marginpar{Ceiling is correct here and floor is correct in the proof of Thm 3.1! The 'reason' for the difference is that we sum over terms with $r^{-\ell}$ in (3.13). Do we want do unify the presentation?}
where $\lceil x \rceil$ is the ceiling function. In the case $M\geq 1$ and $m=1$ we have
\linenopax
\begin{align*}
	r^{\ell}\ge\in\left(0,a_1\right]\ \
	\Leftrightarrow\ \ \ell\in \left[L_1(\ge),\infty\right)\cap\mathbb Z.
\end{align*}
See Fig.~\ref{fig:Lm} and Ex.~\ref{ex:Sierpinski:gasket} in Sec.~\ref{sec:examples} for examples of the form the functions $L_m$ take.
	\begin{figure}[b]
  \centering
		\scalebox{0.55}{\includegraphics{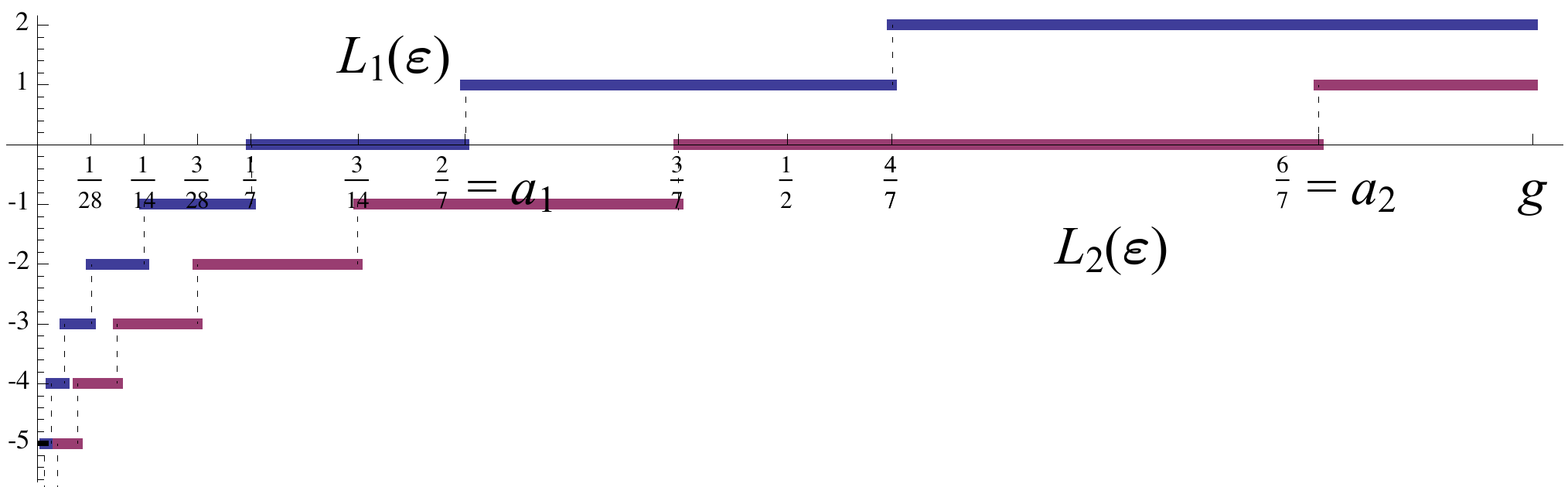}}
		\caption{\captionsize The functions $L_1(\ge)$ and $L_2(\ge)$ as defined in \eqref{eqn:Lm} for an example where $a_1 = \frac27, a_2 = \frac67, r=\frac12,$ and $g=1$. (For clarity, $L_3(\ge)$ is not depicted here.)}
		\label{fig:Lm}
	\end{figure}
Substituting the pluriphase representation \eqref{eqn:pluriphase} into \eqref{eqn:p(e)} and using the funtions $L_m$, we obtain
\linenopax
\begin{align}
	p(\ge)
	&=\ge^{D-d}\sum_{\ell\in\mathbb Z}r^{\ell(D-d)}\left[\sum_{m=1}^M\one_{(a_{m-1},a_m]}(r^{\ell}\ge)\sum_{k=0}^d\kappa_{m,k}(r^{\ell}\ge)^{d-k}+\one_{(g,\infty)}(r^{\ell}\ge)\lambda_d(\Gamma)\right] \notag \\
	&=\ge^{D-d}\Bigg{[}\sum_{\ell=L_1(\ge)}^{\infty}r^{\ell(D-d)}\sum_{k=0}^d\kappa_{1,k}(r^{\ell}\ge)^{d-k}
	+\sum_{m=2}^M\sum_{\ell=L_m(\ge)}^{L_{m-1}(\ge)-1}r^{\ell(D-d)}\sum_{k=0}^d\kappa_{m,k}(r^{\ell}\ge)^{d-k} \notag \\
	&\hstr[10]	+\sum_{\ell=-\infty}^{L_M(\ge)-1}r^{\ell(D-d)}\lambda_d(\Gamma)\Bigg{]}.
	\label{eqn:p(e)-pluriphase-expansion}
\end{align}
Since $\lim_{\ge \to 0^+} \lambda_d(F_{\ge}\cap\Gamma) = 0$, we know that $\gk_{1,d}=0$, but in fact more is true: the strong feasibility of $O$ allows us to again invoke \eqref{eq:FeoutsideSO} and thereby deduce that $\kappa_{1,k}=0$ for all $k \geq D$. This remark is crucial since it ensures absolute convergence of the first series. Since also the other series are absolutely convergent we may change the order of summation and evaluate the series over $\ell$ by means of the geometric series to obtain
\linenopax
\begin{align*}
	p(\ge)
	&= \sum_{k=0}^d\kappa_{1,k} \frac{(r^{L_1(\ge)}\ge)^{(D-k)}}{1-r^{D-k}}
	+ \sum_{k=0}^d\sum_{m=2}^M\kappa_{m,k}\ge^{D-k}\frac{r^{L_m(\ge)(D-k)}-r^{L_{m-1}(\ge)(D-k)}}{1-r^{D-k}} - \lambda_d(\Gamma)\frac{(r^{L_M(\ge)}\ge)^{(D-d)}}{1-r^{D-d}} \\
	&= \sum_{m=1}^{M-1} \sum_{k=0}^d \frac{(r^{L_m(\ge)} \ge)^{D-k}}{1-r^{D-k}} (\kappa_{m,k}-\kappa_{m+1,k})
	+ \sum_{k=0}^d \kappa_{M,k} \frac{(r^{L_M(\ge)} \ge)^{(D-k)}}{1-r^{D-k}}
	- \frac{\lambda_d(\Gamma) (r^{L_M(\ge)}\ge)^{(D-d)}}{1-r^{D-d}}.
\end{align*}
Setting
\linenopax
\begin{align}\label{eqn:kappa-conventions}
	\kappa_{M+1,k} \coloneqq 0 \q \text{for } k=0,\ldots,d-1,
	\qq \text{and} \qq
	\kappa_{M+1,d} \coloneqq \lambda_d(\Gamma),
\end{align}
we may write
\linenopax
\begin{equation}\label{eqn:p(e)-in-terms-of-eta_k}
	p(\ge) = \sum_{k=0}^d \frac{\ge^{D-k}}{1-r^{D-k}} \eta_k(\ge),
\end{equation}
where
\linenopax
\begin{equation}\label{eqn:eta_k}
	\eta_k(\ge) \coloneqq \sum_{m=1}^M r^{L_m(\ge)(D-k)}(\kappa_{m,k}-\kappa_{m+1,k}).
\end{equation}

Our aim is to show that $p$ is non-constant. For this, we restrict our investigations to an interval whose length aligns with the multiplicative period $r$ of $p$, namely the interval $(rg,g]$.  On this interval, each $L_m$ is piecewise constant with at most one point of discontinuity. Therefore each $\eta_k$ is piecewise constant on $(rg,g]$ with at most $M$ points of discontinuity. Moreover, these points of discontinuity coincide for each $k=0,\ldots,d$. Therefore, there is a finite number of disjoint intervals of strictly positive length
\linenopax
\begin{align}\label{eqn:Iq}
	I_1,\ldots,I_Q \qq\text{with}\q (rg,g] = \bigcup_{q=1}^Q I_q,
\end{align}
such that all $\eta_k$ are constant on each $I_q$. We denote the constant value of $\gh_k$ on $I_q$ by $\gb_{k,q}$.

The strict positivity of $p$ on $(rg,g]$ implies, that for each $q\in\{1,\ldots,Q\}$ there exists some $k\in\{0,1,\dots,d\}$ such that $\gb_{k,q} \neq 0$. %Note that there must be an $I_{q_0}$ for which not all $\gb_{k,q_0}$ are zero%
%\footnote{In fact, it can be shown that for every $q$, there is a $k$ with $\gb_{k,q} \neq 0$ (using the approximation $\tilde Z(\ge) \sim p(\ge)$; see \eqref{eqn:Z(t)-sim-sum} and \eqref{eqn:tildeZ}), but this requires a bit more work and is not necessary here.}.
Indeed, if all $\gb_{k,q}$ are zero for some $q$, then obviously $p$ would be identically zero on $I_q$, in contradiction with its positivity.
%Thus, there must be some index $q_0$ such that one of constants $\{\gb_{0,q_0},\dots,\gb_{d,q_0}\}$ is nonzero.

Now fix some $q_0$ and assume $p$ is constant on $I_{q_0}$. Then $p'(\ge)=0$ on the interior  $\inn(I_{q_0})$ of $I_{q_0}$, which would imply
\linenopax
\begin{align}\label{eqn:p'}
	0 = \sum_{k=0}^d \frac{D-k}{1-r^{D-k}} \gb_{k,q_0} \ge^{D-k-1},
	\qq\text{ for each } \ge\in \inn (I_{q_0}).
\end{align}
However, this contradicts the linear independence of the power functions $\{\ge^{D-1}, \dots,\ge^{D-d-1}\}$ on $I_{q_0}$. Hence $p$ is not a constant function and it follows now by Cor.~\ref{cor:subresult} that $F$ is not Minkowski measurable. %In the case when $D = j$ is an integer, we have included the extra hypothesis precisely to ensure that the \nth[j] summand in \eqref{eqn:p'} vanishes and the contradiction remains.
\end{proof}

\subsection{The case when $D = \dim_\sM F$ is an integer.} 

When the Minkowski dimension $D$ of $\attr$ is an integer then both are possible: $\attr$ can be Minkowski measurable or non-Minkowski measurable. In our main theorem of this section, Thm.~\ref{thm:pluriphase-integer-result}, we provide equivalent characterizations of Minkowski measurability in the current setting.
 
	Suppose that $F$ is pluriphase with respect to $\Gamma$. We consider the functions $L_m$ as defined in \eqref{eqn:Lm}. Examples of these functions appear in Fig.~\ref{fig:Lm} and Fig.~\ref{fig:Lm2}. In the proof of Thm.~\ref{thm:pluriphase-result}, we used the facts that each $L_m$ is piecewise constant on  $(rg,g]$ with at most one point of discontinuity to deduce that there is a finite number of disjoint intervals $\{I_q\}_{q=1}^Q$ of strictly positive length such that $(rg,g] = \bigcup_{q=1}^Q I_q$ and such that every $L_m$ is constant on each interval $I_q$ ; see \eqref{eqn:Iq}.

	For the formulation of one of the equivalent characterizations of Minkowski measurability in Thm.~\ref{thm:pluriphase-integer-result}, it is convenient to group together those indices $m$ for which $L_m$ has the same point of discontinuity in $(rg,g]$:
\linenopax
\begin{align}
	U_q &\defeq \Big{\{}m\in\{1,\ldots,n\}\suth  L_m(\ge_q)\neq L_m(\ge_{q+1})\ \text{for}\ \ge_q\in I_q,\ \ge_{q+1}\in I_{q+1}\Big{\}},\q q<Q,\label{eqn:Uq}\\
	U_Q &\defeq \{1,\ldots,M\}\setminus\bigcup_{q=1}^{Q-1}U_q.
		\label{eqn:UQ}
\end{align}
Note that for each $m\in U_Q$ the function $L_m$ is a constant function on $(rg,g]$.

We now give a precise statement of part \ref{it:intro:integer} of Thm.~\ref{thm:main-result}, that is for the case when $D$ is an integer.

\begin{theorem}\label{thm:pluriphase-integer-result}
Let $F\subset\bR^d$ be a self-similar set which is the attractor of the lattice self-similar system $S=\{S_1,\dots,S_N\}$, $N\geq 2$ satisfying the OSC. Assume the Minkowski dimension $D = \dim_\sM F$ is an integer different from $d$ and that we can find a strong feasible open set $O$ satisfying the projection condition (see Def.~\ref{def:metproj}) such that $F$ is pluriphase with respect to $\Gamma(O)$. Let $C>0$.	
Then the following assertions are equivalent:
\begin{enumerate}[label=(\arabic*)]
	\item\label{it:Mmb} $F$ is Minkowski measurable with Minkowski content given by
	\begin{equation}
		\sM_D(F)=\frac{\ln r}{\sum_{i=1}^N r_i^D \ln r_i} \cdot C\,.
		%\q\text{and}\q
		%C = \sum_{m\in M_Q} L_m(\ge)(\kappa_{m+1,D}-\kappa_{m,D}).
	\end{equation}
 % with Minkowski content $c\cdot\tfrac{\ln r}{\sum_{i=1}^N r_i^D \ln r_i}$.
	\item\label{it:constant} The function $p$ from \eqref{eqn:p(e)} is a constant function taking the value $C$.
	\item\label{it:algebraicLmwithoutq} For $\ge\in(rg,g]$,
		\linenopax
		\begin{align}%\label{eq:integeriffconstantLmwithoutq}
			C &= \sum_{m=1}^{M} L_m(\ge)(\kappa_{m+1,D}-\kappa_{m,D})
			\q\text{and}\q \label{eqn:DintegerC1}\\
			0 &= \sum_{m=1}^{M} r^{L_m(\ge)(D-k)}(\kappa_{m,k}-\kappa_{m+1,k})\quad\text{for } k\neq D .\label{eqn:Dintegernull1}
		\end{align}
	\item\label{it:algebraicaq} With $U_q$ defined as in \eqref{eqn:Uq}--\eqref{eqn:UQ} and $g$ as in \eqref{eqn:g},
          \linenopax
	\begin{equation}\label{eqn:Dintegernull2}
            0=\sum_{m\in U_q}a_m^{D-k}(\kappa_{m,k}-\kappa_{m+1,k})
          \end{equation}
          for all pairs $(k,q)\in(\{0,\ldots,d\}\times\{1,\ldots,Q\})\setminus\{(D,Q)\}$, and
          \begin{equation}\label{eqn:DintegerC2}
            C=\sum_{m\in U_Q}L_m(g)(\kappa_{m+1,D}-\kappa_{m,D}).
          \end{equation}
          
	\end{enumerate}
	Moreover, if $O$ is such that $F$ is monophase w.r.t.\ $\Gamma$, then these assertions are never met and so in particular $F$ is not Minkowski measurable.
%	\begin{equation}
%		\sM_D(F)=\frac{\ln r}{\sum_{i=1}^N r_i^D \ln r_i} C \cdot
%		%\q\text{and}\q
%		%C = \sum_{m\in M_Q} L_m(\ge)(\kappa_{m+1,D}-\kappa_{m,D}).
%	\end{equation}
%
\end{theorem}

%\begin{theorem}\label{thm:pluriphase-integer-result}
%	Let $F$ be a self-similar set which is the attractor of the lattice self-similar system $S=\{S_1,\dots,S_N\}$, $N\geq 2$ and for which the Minkowski dimension $D = \dim_\sM F$ is an integer, i.e., $D \in \{0,1,2,\dots,d-1\}$.
%	Assume that $F$ satisfies the OSC and that we can find a strongly feasible open set $O$ which satisfies the following additional conditions (but see Rem.~\ref{rem:hypotheses-clarification}):
%	\begin{enumerate}[label=(\roman*)]
%		\item (Projection condition; see Def.~\ref{def:metproj}) $S_i O \subseteq \overline{\pi_{F}^{-1}(S_iF)}$ for $i=1,\ldots,N$, where $\pi_F$ denotes the metric projection onto $F$,
%		\item (Pluriphase; see Def.~\ref{def:pluriphase}) \;
%		 $O$ is pluriphase.
%	\end{enumerate}
%	We further assume that there is a $k \in \{0,\dots,d\}$ and a $q \in \{1,\dots,Q\}$ such that
%	\begin{equation}\label{eq:integeriffconstantwoc}
%		\sum_{m\in U_q}a_m^{D-k}(\kappa_{m,k}-\kappa_{m+1,k}) \neq 0,
%	\end{equation}
%	where $U_q$ are defined as in \eqref{eqn:Uq}--\eqref{eqn:UQ}.
%	
%	In this situation, $F$ is not Minkowski measurable.
%\end{theorem}

\begin{remark}\label{def:pluriscaling}
	Nonmeasurability arising from lattice-type iterated function systems is due to geometric oscillations arising from the alignments of multiplicative periods in the scaling factors of the  various mappings in the IFS; see \cite{FGCD}. In some sense, the algebraic conditions formulated in \ref{it:algebraicLmwithoutq} and more clearly in \ref{it:algebraicaq} describe the situation when there are extra geometric oscillations induced by the pluriphase representation of the volume function $\lambda_d(F_\ge\cap \Gamma)$ that cancel out the geometric oscillations intrinsic to the IFS. These conditions should be viewed as a kind of \emph{latticeness} of the representation (and thus of the set $\Gamma$ in relation with $\attr$).
\end{remark}

%\marginpar{To do} [Add some more explanations here regarding this result]

\begin{remark}\label{rem:hypotheses-clarification-2}
	Note that the hypothesis $D \neq d$ is equivalent to the hypothesis that $F$ is nontrivial, by Prop.~\ref{cor:OSC-dimension-d-implies-trivial}.
\end{remark}

\begin{proof}[Proof of Thm.~\ref{thm:pluriphase-integer-result}]
The equivalence of the assertions \ref{it:Mmb} and \ref{it:constant} is clear from Cor.~\ref{cor:subresult}.
To show the equivalence \ref{it:constant} $\Leftrightarrow$ \ref{it:algebraicLmwithoutq},
we go again through the steps of the proof of Thm.~\ref{thm:pluriphase-result} and point out the modifications necessary in the case when $D$ is an integer. Up to equation \eqref{eqn:p(e)-pluriphase-expansion} there is no difference in the derivation, but from this equation onwards, the $D^{\textrm{th}}$ terms in all the summations over $k$ have to be treated differently. First, notice that
%notice some slight differences. modifications. start by repeating the derivation from \eqref{eqn:p(e)-pluriphase-expansion}--\eqref{eqn:eta_k} in the proof of with slight modifications. First, notice that
$\gk_{1,D}=0$ according to the discussion just after \eqref{eqn:p(e)-pluriphase-expansion} (since $D \leq D$). So there is no concern regarding the summation from $L_1(\ge)$ to $\iy$ for the $D^{\textrm{th}}$ terms in the first summand in \eqref{eqn:p(e)-pluriphase-expansion}, they just vanish. The second summand in \eqref{eqn:p(e)-pluriphase-expansion} changes; we rewrite it here with the term for $k=D$ extracted:
\linenopax
\begin{align*}%\label{eqn:}
	\ge^{D-d}\sum_{m=2}^M\sum_{\ell=L_m(\ge)}^{L_{m-1}(\ge)-1}r^{\ell(D-d)}\sum_{\substack{k=0 \\k \neq D}}^d\kappa_{m,k}(r^{\ell}\ge)^{d-k} + \ge^{D-d}\sum_{m=2}^M\sum_{\ell=L_m(\ge)}^{L_{m-1}(\ge)-1}r^{\ell(D-d)} \kappa_{m,D}(r^{\ell}\ge)^{d-D}
\end{align*}
The first expression is dealt with exactly as in the proof of Thm.~\ref{thm:pluriphase-result}. For the second term, we have
\linenopax
\begin{align*}%\label{eqn:}
	\gh_D(\ge)
	&\coloneqq \ge^{D-d}\sum_{m=2}^M\sum_{\ell=L_m(\ge)}^{L_{m-1}(\ge)-1}r^{\ell(D-d)} \kappa_{m,D}(r^{\ell}\ge)^{d-D}
	= \sum_{m=2}^M\sum_{\ell=L_m(\ge)}^{L_{m-1}(\ge)-1} \kappa_{m,D}
	= \sum_{m=1}^M L_m(\ge) (\kappa_{m+1,D} - \kappa_{m,D})\,,
\end{align*}
where we have used the notation introduced in \eqref{eqn:kappa-conventions}.
Thus we obtain the following modification of \eqref{eqn:eta_k}
\linenopax
\begin{equation}\label{eqn:eta_k-and-eta_D}
	\eta_k(\ge)\defeq\begin{cases}
	\sum_{m=1}^M r^{L_m(\ge)(D-k)}(\kappa_{m,k}-\kappa_{m+1,k}),   & k \neq D,\\
	\sum_{m=1}^M L_m(\ge)(\kappa_{m+1,k}-\kappa_{m,k}),   & k = D,
\end{cases}
\end{equation}
and \eqref{eqn:p(e)-in-terms-of-eta_k} is replaced by
\linenopax
\begin{equation}\label{eqn:p(e)-in-terms-of-eta_j}
	p(\ge) = \eta_D(\ge) + \sum_{\substack{k=0\\k \neq D}}^d \frac{\ge^{D-k}}{1-r^{D-k}} \eta_k(\ge),
\end{equation}
where still all the functions $\eta_k(\ge)$ are piecewise constant with finitely many pieces in $(rg,g]$.

Now assume that \ref{it:constant} holds, i.e.\ $p(\ge)=C$ for $\ge\in(rg,g]$ and some $C>0$. Restricting to an arbitrary subinterval on which the functions $\eta_k$ are all constant, we can use again the linear independence of the functions $\ge^{D-k}$, $k=0,\ldots,d$ to conclude that $\eta_k(\ge)=0$ for $k\neq D$ and thus $\eta_D(\ge)=C$ on this subinterval. Since this applies to all such subintervals, it holds on $(rg,g]$, which shows that \ref{it:constant} implies \ref{it:algebraicLmwithoutq}. The reverse implication is obvious from equation \eqref{eqn:p(e)-in-terms-of-eta_j}.
%
%
%a Using again the  Our aim is to show that $p$ is non-constant, so by way of contradiction, we assume $p(\ge) = c$ on $(rg,g]$. Now \eqref{eqn:p(e)-in-terms-of-eta_j} and \eqref{eqn:eta_k-and-eta_D} imply the following equivalence on $(rg,g]$:
%\linenopax
%\begin{align}\label{eqn:p-const-iff-eta=0}
%	p(\ge)=c
%	\qq \iff \qq
%	\gh_D(\ge) = c \q\text{and}\q \gh_k(\ge) = 0, \text{ for $k \neq D$}.
%\end{align}

Our next step is to show the equivalence of \ref{it:algebraicLmwithoutq} and \ref{it:algebraicaq}.

%We next show that \eqref{eqn:p-const-iff-eta=0} holds if and only if
%\linenopax
%	\begin{equation}\label{eq:integeriffconstantwoc}
%		\sum_{m\in U_q}a_m^{D-k}(\kappa_{m,k}-\kappa_{m+1,k}) = 0,
%		\qq \text{for all}\q  k=0,\ldots,d \text{ and } q=1,\ldots,Q.
%	\end{equation}
	%\marginpar{improve rest of proof!?}
(i) First, we consider the case $k=D$ and show that \eqref{eqn:DintegerC2} together with \eqref{eqn:Dintegernull2} holding for pairs $(k,q)\in\{D\}\times\{1,\ldots,Q-1\}$ is equivalent to \eqref{eqn:DintegerC1}.

Observe that \eqref{eqn:DintegerC1} %$\gh_D(\ge) = C$
can be rewritten as 
	\linenopax
	\begin{align}\label{eq:cAq}
		C	&= \sum_{m=1}^M  L_m(\ge)(\kappa_{m+1,D}-\kappa_{m,D})
			 = \sum_{q=1}^Q \underbrace{\sum_{m\in U_q}  L_m(\ge)(\kappa_{m+1,D}-\kappa_{m,D})}_{\eqdef A_{D,q}(\ge)}
	\end{align}
	for all $\ge\in(rg,g]$. %, where $U_q$ is as defined in \eqref{eqn:Uq}--\eqref{eqn:UQ}.
 Note that for $q=Q$ the sum $A_{D,Q}(\ge)=A_{D,Q}(g)$ is independent of $\ge\in(rg,g]$ by construction. Thus, if $Q=1$ then \eqref{eqn:DintegerC1} is equivalent to \eqref{eqn:DintegerC2}. Now consider the case $Q\geq 2$ and fix $q\in\{1,\ldots,Q-1\}$.
By construction $A_{D,q'}$ is constant on $I_q\cup I_{q+1}$ when $q'\in\{1,\ldots,Q-1\}\setminus\{q\}$. Moreover, for $m\in U_q$ we have $L_m(\ge_{q+1})-L_m(\ge_q)=1$ for $\ge_{q+1}\in I_{q+1}$ and $\ge_q\in I_q$, yielding
\linenopax
\begin{equation}
	 \left\lvert A_{D,q}(\ge_{q+1})-A_{D,q}(\ge_q)\right\rvert= \left\lvert\sum_{m\in U_q} (\kappa_{m+1,D}-\kappa_{m,D})\right\rvert.
\end{equation}
Thus, \eqref{eq:cAq} holds %$\sum_{q=1}^{Q} A_{D,q}(\ge) \equiv c$ on $(rg,r]$
if and only if
	\linenopax
\begin{align*}
	0 &= \sum_{m\in U_q} (\kappa_{m+1,D}-\kappa_{m,D}) \q \text{for } q < Q,
	\qq\text{and}\qq
	C = \sum_{m\in U_Q} L_m(\ge) (\kappa_{m+1,D}-\kappa_{m,D})
\end{align*}
for $\ge\in(rg,g]$. Noting that $L_m(\ge)=L_m(g)$ holds for all $\ge\in(rg,g]$ whenever $m\in U_Q$, the statement (i) is verified.

(ii) Second, we consider the case $k \neq D$ and show that \eqref{eqn:Dintegernull2} holding for pairs $(k,q)\in(\{0,\ldots,d\}\setminus\{D\})\times\{1,\ldots,Q\}$ is equivalent to \eqref{eqn:Dintegernull1}.
 
Observe that, $\eta_k(\ge)=0$ on $(rg,g]$ if and only if
\linenopax
\begin{align*}
	0 = \gh_k(\ge)
	= \sum_{q=1}^Q \underbrace{\sum_{m\in U_q} r^{L_m(\ge)(D-k)}(\kappa_{m,k}-\kappa_{m+1,k})}_{\eqdef A_{q,k}(\ge)}, \q \text{on $(rg,g]$}.
\end{align*}
In the same way as in the case of $k=D$, one can deduce that $0=\sum_{q=1}^Q A_{q,k}(\ge)$ on $(rg,g]$ if and only if
	\linenopax
\begin{align}\label{eq:0Aqk}
	0=A_{q,k}(\ge)=\sum_{m\in U_q} r^{L_m(\ge)(D-k)}(\kappa_{m,k}-\kappa_{m+1,k})
	\q \text{on $(rg,g]$, \q for $q=1,\ldots,Q$}.
\end{align}
By definition of the sets $U_q$, we know that $\{\log_r\tfrac{a_m}{\ge}\}=\{\log_r\tfrac{a_{m'}}{\ge}\}$ for $m,m'\in U_q$. Denoting this common value by $g_q(\ge)$, Equation \eqref{eq:0Aqk} is equivalent to
	\linenopax
\begin{align*}
	0 = r^{(1-g_q(\ge))(D-k)}\sum_{m\in U_q}\left(\frac{a_m}{\ge}\right)^{D-k} (\kappa_{m,k}-\kappa_{m+1,k})
	\q \text{on $(rg,g]$, \q for $q=1,\ldots,Q$}.
\end{align*}
This is equivalent to \eqref{eqn:Dintegernull2}, since $k\neq D$ and the power functions $\{\ge^{-D},\ldots,\ge^{d-D}\}$ are linearly independent. Thus, \ref{it:algebraicLmwithoutq} is equivalent to assertion \ref{it:algebraicaq}. %\eqref{eq:integeriffconstantwoc}.
%\linenopax
%\begin{equation}
%	0 = \sum_{m\in U_q}a_m^{D-k} (\kappa_{m,k}-\kappa_{m+1,k}),
%	\q \text{for $q=1,\ldots,Q$}.
%\end{equation}

Finally, if $F$ is monophase w.r.t.\ $\Gamma$, then $M=1$. 
The discussion directly after \eqref{eqn:p(e)-pluriphase-expansion} gives $\kappa_{1,d}=0$, and \eqref{eqn:kappa-conventions} gives $\kappa_{2,d}=\lambda_d(\gG)\neq0$. Therefore, \eqref{eqn:Dintegernull1} cannot be satisfied for $k=d$
and the last assertion in Thm.~\ref{thm:pluriphase-integer-result} follows. %from the observation that if $M=1$ then \eqref{eqn:DintegerC1} is not satisfied.
\end{proof}

\section{Examples}\label{sec:examples}

In the case $D\in\mathbb N$ there are examples of self-similar sets arising from lattice IFS which are Minkowski measurable. In Ex.~\ref{ex:square} we provide a simple example of a lattice IFS with common scaling ratio $r=\frac12$ for which the attractor has integer Minkowski dimension $D = \dim_\sM F = 2$. Note that this example comes by ``cheating'' the nontriviality condition via embedding in a higher-dimensional Euclidean space; see Rem.~\ref{thm:trivial-issues}.

%On the other hand there are examples of such sets which are not Minkowski measurable (see ???). %In the following we establish an algebraic condition involving the coefficients $\kappa_{m,k}$ and the partition points $a_m$  of the pluriphase condition which is equivalent to Minkowski measurability of the associated self-similar set.

\begin{ex}\label{ex:square}
	Let $S_1,\ldots,S_4\colon\mathbb R^3\to\mathbb R^3$ be given by
	\linenopax
        \begin{equation*}
	\begin{array}{rlrl}
	S_1(x)&\hspace{-0.7em}\defeq x/2,   &   S_3(x)&\hspace{-0.7em}\defeq x/2+ (0,1/2,0),\\
	S_2(x)&\hspace{-0.7em}\defeq x/2+ (1/2,0,0), \qquad\qquad   &   S_4(x)&\hspace{-0.7em}\defeq x/2+ (1/2,1/2,0).
	\end{array}
        \end{equation*}
	It is not difficult to see that $F\defeq [0,1]\times [0,1]\times \{0\}$ is the associated invariant set and that $D=2$, so we are in the case $D\in\mathbb N$.
	Define $O\defeq(0,1)\times (0,1) \times (-1/2,1/2)$. Then $O$ is a strong feasible open set for $\{S_1,\ldots,S_4\}$. Moreover the nontriviality condition is satisfied, since $\bigcup_{i=1}^4 \overline{S_i(O)}=[0,1]\times[0,1]\times[-1/4,1/4]$.
 With $\Gamma\defeq O\setminus \mathbf{S}O$ as before, % and note that $r=1/2$, $D=2$, $d=3$. Then our function $p$ becomes
%	\begin{align}
%	p(\ge)=\ge^{D-d}\sum_{\ell\in\mathbb Z}r^{\ell(D-d)}\lambda_d(F_{r^{\ell}\ge}\cap \Gamma)
%	=\ge^{-1}\sum_{\ell\in\mathbb Z}2^{\ell}\lambda_3(F_{2^{-\ell}\ge}\cap \Gamma).	
%	\end{align}
%	We have
	\begin{equation}
	\lambda_3(F_{\ge}\cap \Gamma)=
	\begin{cases}
	0, & \ge<1/4,\\
	2(\ge-1/4), & 1/4\leq\ge<1/2,\\
	1/2, & 1/2\leq \ge.
	\end{cases}
	\end{equation}
	Thus, $a_0\defeq 0$, $a_1=1/4$ and $a_2\defeq g=1/2$, which implies $\{\log_r (a_m)\}=0$ for $m=1,2$, since $r=1/2$.
        Hence, $Q=1$ with $Q$ as in \eqref{eqn:Iq}, implying $U_Q=\{1,2\}$. Moreover, $\kappa_{2,2}=2$, $\kappa_{2,3}=-1/2$, $\kappa_{3,3}=1/2$ and $\kappa_{m,k}=0$ for all other pairs $m,k$.
	Using Thm.~\ref{thm:pluriphase-integer-result} \ref{it:algebraicaq}, we conclude that $F$ is Minkowski measurable with Minkowski content 2.
	This can also be deduced directly by evaluating the function $p$.
\end{ex}

\begin{ex}\label{ex:Sierpinski:gasket}
  In this example, we return to the example which we provided in Fig.~\ref{fig:examplesnew}. 
  	\begin{figure}[t]
  \centering
		\scalebox{0.75}{\includegraphics{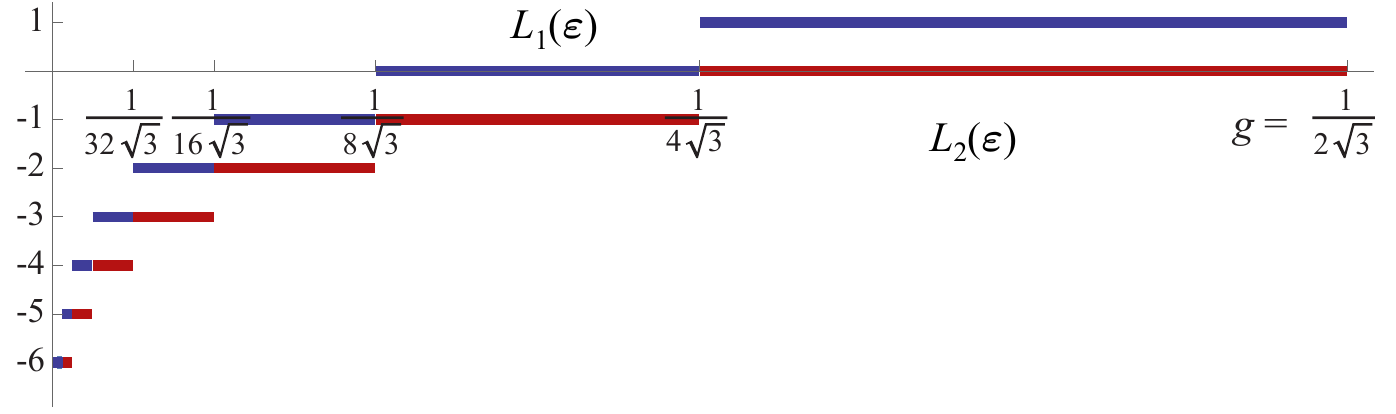}}
		\caption{\captionsize The functions $L_1(\ge)$ and $L_2(\ge)$ as defined in \eqref{eqn:Lm} for %an example where $a_1 = \frac27, a_2 = \frac67, r=\frac12,$ and $g=1$.
	the Sierpinski gasket with $O$ and $\gG$ as in Fig.~\ref{fig:examplesnew}, see Example~\ref{ex:Sierpinski:gasket} for details. Here, $a_1=\sqrt{3}/12$, $a_2=g=\sqrt{3}/6$ and $r=1/2$.}
		\label{fig:Lm2}
	\end{figure}
	Here, $S=\{S_1,S_2,S_3\}$ is the standard IFS which generates the Sierpinski gasket $\attr$, and $O$ and $\gG$ are as in Fig.~\ref{fig:examplesnew}. Then,
\begin{equation}\label{eq:Sierpgasket:polynom}
  \lambda_2(\attr_{\ge}\cap\gG)=
  \begin{cases}
    6\sqrt{3}\ge^2 &\colon 0\leq \ge<\sqrt{3}/12\\
    6\sqrt{3}\ge^2-3\ge+\sqrt{3}/4 &\colon \sqrt{3}/12\leq\ge<\sqrt{3}/6\\
    \sqrt{3}/4 &\colon \sqrt{3}/6\leq\ge.
  \end{cases}
\end{equation}
Thus, $\attr$ is pluriphase w.r.t.\ $\gG$. Moreover, $D=\dim_{\mathcal M}(\attr)=\log_2(3)\notin\mathbb N$ and $O$ is a strong feasible open set satisfying the projection condition. Therefore, we can apply Thm.~\ref{thm:pluriphase-result} and deduce that $\attr$ is not Minkowski measurable. 
For this example we want to visualize the functions $L_m$ from \eqref{eqn:Lm}, which the proof of Thm.~\ref{thm:pluriphase-result} heavily uses. From \eqref{eq:Sierpgasket:polynom} we conclude that $a_0\defeq 0$, $a_1=\sqrt{3}/12$ and $a_2\defeq g=\sqrt{3}/6$. Moreover, $r=1/2$. Therefore,
\linenopax
\begin{align*}
  L_1(\ge)&\defeq\left\lceil\log_r\left(\tfrac{a_1}{\ge}\right)\right\rceil
  =\left\lceil\tfrac{\ln(\sqrt{3}/12)-\ln(\ge)}{-\ln(2)}\right\rceil\q\text{and}\q
  L_2(\ge)&=\left\lceil\tfrac{\ln(\sqrt{3}/6)-\ln(\ge)}{-\ln(2)}\right\rceil.
\end{align*}
The plot of $L_1$ and $L_2$ is provided in Fig.~\ref{fig:Lm2}.
\end{ex}

\pgap

\bibliographystyle{amsalpha}
\bibliography{strings}

\end{document}